\documentclass[10pt,a4paper]{article}

\usepackage[utf8]{inputenc}
\usepackage{amsfonts} 
\usepackage{caption} 
\usepackage{mathtools} 
\usepackage{amsthm}
\usepackage{amsmath,amssymb} 
\usepackage{graphicx} 
\usepackage[colorlinks=false]{hyperref} 
\usepackage{mathbbol} 
\usepackage[title]{appendix} 
\usepackage{tikz-cd}  
\usepackage{comment}
\usepackage{float} 
\newtheorem{theorem}{Theorem}[section] 

\newtheorem{definition}[theorem]{Definition} 
\newtheorem{remark}[theorem]{Remark}

\newtheorem{proposition}[theorem]{Proposition}
\newtheorem{corollary}[theorem]{Corollary}
\newtheorem{lemma}[theorem]{Lemma}

\newtheorem{notation}[theorem]{Notation} 

\newtheorem{assump}{Assumption}

\usepackage{cleveref}
\usepackage[small]{titlesec} 

\newcommand{\footremember}[2]{%
	\footnote{#2}
	\newcounter{#1}
	\setcounter{#1}{\value{footnote}}%
}
\newcommand{\footrecall}[1]{%
	\footnotemark[\value{#1}]%
} 

\makeatletter
\def\blfootnote{\gdef\@thefnmark{}\@footnotetext}
\makeatother

\newcommand{\kNN}{\operatorname{kNN}}

\newcommand{\N}{{\mathbf N}}

\renewcommand{\P}{{\mathbf P}}

\newcommand{\R}{{\mathbb R}}
\newcommand{\ind}{{\mathbf 1}}

\renewcommand\path{\mathfrak{P}}

\newcommand\lbrac{[ \! [}
\newcommand\rbrac{] \! ]}

\newcommand\cA{\mathcal{A}}
\newcommand\cB{\mathcal{B}}
\newcommand\cC{\mathcal{C}}

\newcommand\cF{\mathcal{F}}
\newcommand\cH{\mathcal{H}}

\newcommand\cM{\mathcal{M}}

\newcommand\dR{\mathbb{R}}

\newcommand\bg{\mathbf{g}}

\newcommand\dd{\mathrm{d}}

\newcommand\lapM{\Delta_\cM}

\newcommand{\ABS}[1]{{{\left| #1 \right|}}} 
\newcommand{\SCA}[1]{{{\left<#1\right>}}} 
\newcommand{\NRM}[1]{{{\left\| #1\right\|}}} 
\newcommand{\PAR}[1]{{{\left(#1\right)}}} 
\newcommand{\SBRA}[1]{{{\left[#1\right]}}} 
\renewcommand{\leq}{\leqslant}
\renewcommand{\geq}{\geqslant}

\def\toan#1{\color{red}#1 \color{black}}

\begin{document}
\title{Strong uniform convergence of Laplacians of random geometric and directed kNN graphs on compact manifolds}

\author{H\'el\`ene Gu\'erin \footremember{trailer}{Département de Mathématiques, Université du Québec à Montréal (UQAM), Montréal, QC H2X 3Y7, Canada.}
\and
Dinh-Toan Nguyen \footremember{alley}{LAMA, Univ Gustave Eiffel, Univ Paris Est Creteil, CNRS, F-77454 Marne-la-Vallée, France.} \footrecall{trailer}
	\and
Viet Chi Tran \footrecall{alley} \footremember{crm}{IRL 3457, CRM-CNRS, Université de Montréal, Canada}
}
\date{\today}

\maketitle

\begin{abstract}
Consider $n$ points independently sampled from a density $p$ of class {$\cC^2$} on a smooth compact $d$-dimensional sub-manifold $\cM$ of $\R^m$, and consider the generator of a random walk visiting these points according to a transition kernel $K$. We study the almost sure uniform convergence of this operator to the diffusive Laplace-Beltrami operator when $n$ tends to infinity. This work extends known results of the past 15 years. In particular, our result does not require the kernel $K$ to be continuous, which covers the cases of walks exploring $k$NN-random and geometric graphs, and convergence rates are given. The distance between the random walk generator and the limiting operator is separated into several terms: a statistical term, related to the law of large numbers, is treated with concentration tools and an approximation term that we control with tools from differential geometry. The convergence of $k$NN Laplacians is detailed.
\end{abstract}

\noindent \textit{MSC2020 subject classifications:} 60F05, 05C81, 62G05. \\

\noindent \textit{Key words and phrases:} Laplace-Beltrami operator, {graph Laplacian}, kernel estimator, large sample approximations, Riemannian geometry, concentration inequalities, k-nearest-neighbors.\\

\blfootnote{\textit{Acknowledgements:} The research of H.G. and D.T.N. is part of the Mathematics for Public Health program funded by the Natural Sciences and Engineering Research Council of Canada (NSERC). H.G. is also supported by NSERC discovery grant (RGPIN-2020-07239). D.T.N. and V.C.T. are supported by Labex B\'ezout (ANR-10-LABX-58), GdR GeoSto 3477 and by the European Union (ERC-AdG SINGER-101054787).}


\section{Introduction}

Let $\mathcal{M}$ be a compact smooth $d$-dimensional submanifold without boundary of $\mathbb{R}^m$, {which} we embed with the Riemannian structure induced by the ambient space $\mathbb{R}^m$. Denote by $\| \cdot \|_2$, $\rho(\cdot,\cdot)$ and $\mu(\dd x)$ respectively the Euclidean distance of $\mathbb{R}^m$, the geodesic distance on $\cM$ and the volume measure on $\mathcal{M}$. Let $(X_i, i \in \mathbb{N})$ be a sequence of i.i.d. points in $\mathcal{M}$ sampled from the distribution $p(x)\mu(\mathrm{d}x)$, where $p \in \mathcal{C}^2$ is a continuous function such that $p(x)\mu(\mathrm{d}x)$ defines a probability measure on $\mathcal{M}.$\\
In this article, we study the limit of the random operators $( \mathcal{A}_{h_n,n} , n \in \mathbb{N})$:
\begin{equation}\label{eq:def_Ahn-n}
	\mathcal{A}_{h_n,n}(f)(x):=\frac{1}{nh_n^{d+2}} \sum_{i=1}^nK\left( \frac{ \| x-X_i\|_2}{h_n} \right)(f(X_i)-f(x)), \quad x \in \mathcal{M}
\end{equation}
 where $K: \mathbb{R}_+\rightarrow \mathbb{R}_+$ is a function of bounded variation and $(h_n, n \in \mathbb{N})$ is a sequence of positive real numbers converging to 0.\\
 Such operators can be viewed as the infinitesimal generator of continuous time random walks visiting the points $(X_i)_{i\in \lbrac 1,n\rbrac}$, where $\lbrac 1,n\rbrac = \{1,\dots n\}$. The process jumps from its position $x$ to the new position $X_i$ at a rate $K(\|x-X_i\|_2/h_n)/(nh_n^{d+2})$ that depends on the distance between $x$ and $X_i$. Notice that here, the Euclidean distance is used. When walking on the manifold $\cM$, using the geodesic distance and considering the operator 
 \[\widetilde{\mathcal{A}}_{h_n,n}(f)(x):=\frac{1}{nh_n^{d+2}} \sum_{i=1}^nK\left( \frac{ \rho( x,X_i)}{h_n} \right)(f(X_i)-f(x)), \quad x \in \mathcal{M}\]could be also very natural. In fact, for smooth manifolds, the limits of the two operators $\mathcal{A}_{h_n,n}$ and $\widetilde{A}_{h_n,n}$ are the same, as is indicated by \cite[Prop. 2]{Trillos2020}. In view of applications to manifold learning, when $\cM$ is unknown and when only the sample points $X_i$'s are available, using the norm of the ambient space $\R^m$ can be justified.\\
 The operator \eqref{eq:def_Ahn-n} can also be seen as a
 graph Laplacian for a weighted graph with vertices being data points and their convergence has been studied extensively in machine learning literature to approximate the Laplace-Beltrami operator of $\mathcal{M}$   (see e.g. \cite{Singer2006,Gine2006,Hein2007,Belkin2008,Belkin2009,taoshi2020}).  Nonetheless, most of these results are done for Gaussian kernel, i.e., $K(x)=e^{-x^2}$, or sufficiently smooth kernels. These assumptions are too strong to include the case of $\varepsilon$-geometric graphs or $k$-nearest neighbor graphs (abbreviated as $k$NN), and that correspond to choices of indicators for the kernel $K$. In recent years, many works had been done to relax the regularity of $K$ and gave birth to many interesting papers (e.g. \cite{Calder2022,Ting2010}), as discussed below. \\

 In the sequel, under a mild assumption on $K$ (weaker than continuity, see Assumption \ref{hyp:K} below) and a condition on the rate of convergence of $(h_n)$, we show that  almost surely, the sequence of operators $(\mathcal{A}_{h_n,n})$ converges {uniformly on $\cM$} to the second order differential operator $\mathcal{A}$ on $\mathcal{M}$ defined  as
\begin{equation}
	\label{Equation: differential operator}
	\mathcal{A}(f) := c_0\bigg(  \langle \nabla_{\mathcal{M}}(p) , \nabla_{\mathcal{M}}(f) \rangle + \frac{1}{2}p\Delta_{\mathcal{M}}(f)\bigg),
\end{equation}
for all function $f \in \mathcal{C}^{3}(\mathcal{M})$, where $\nabla_{\mathcal{M}}$ and $\Delta_{\mathcal{M}}$ are respectively the gradient operator and Laplace-Beltrami operator of $\mathcal{M}$ (introduced in Section \ref{Section: Some geometric background}) and \begin{equation}
	c_0:=\frac{1}{d}\int_{\mathbb{R}^d} K \left(\|v\|_2 \right)  \| v\|_2^2 \mathrm{d}v {=\frac{1}{d}S_{d-1} \int_0^{\infty} K(a)a^{d+1} \mathrm{d}a},
\end{equation}where $S_{d-1}$ denotes the volume of the unit sphere of $\R^d$.
Moreover, a convergence rate is also deduced, as stated in our main Theorem below (Theorem \ref{Theorem: main theorem}) that we will present after having enounced the assumptions needed on the kernel $K$:
\begin{assump}\label{hyp:K}
The kernel $K\ :\ \mathbb{R}_+\rightarrow \mathbb{R}_+$ is a measurable function with $K(\infty)=0$ and of bounded variation $H$ such that: 
\begin{equation}
    \int_{0}^{\infty} a^{d+3} \, d H(a) < \infty. \label{Assumption: A1}
\end{equation}
\end{assump}
Recall that the total variation $H$ of a kernel $K$ is defined for each nonnegative number $a$ as $H(a)=\sup \sum_{i=1}^n |K(a_i)-K(a_{i-1})|$, where the supremum ranges over all $n\in \mathbb{N}$ and all subdivisions $0=a_0<\cdots <a_n=a$ of $[0,a]$. Assumption \ref{hyp:K} is the key to avoid making continuity hypotheses on the kernel $K$. 

\begin{theorem}[Main theorem]\label{Theorem: main theorem}
Suppose that the density of points $p(x)$ on the compact smooth manifold $\mathcal{M}$ is of class {$\cC^2$}. 
Suppose that Assumptions \ref{hyp:K} for the kernel $K$ are satisfied and that:
\begin{equation}\label{hyp:h_th11}
    \lim_{n\rightarrow +\infty} h_n =0,\qquad \mbox{ and }\qquad \lim_{n\rightarrow +\infty}\frac{\log h_n^{-1}}{nh_n^{d+2}} =0.
\end{equation}
Then, with probability $1$,  for all  $f \in \mathcal{C}^{3}(\mathcal{M})$,
\begin{equation}
	\sup_{x \in \mathcal{M}} \left|  \mathcal{A}_{h_n,n}(f)(x)- \mathcal{A}(f)(x)\right|= O\left(  \sqrt{ \frac{ \log h_n^{-1}}{nh_n^{d+2}}}  + h_n \right) .\label{eq:cv_th}
\end{equation} 
\end{theorem}

{Notice that the window $h_n$ that optimizes the convergence rate in \eqref{eq:cv_th} is of order $n^{-1/(d+4)},$ up to $\log$ factors,
resulting in a convergence rate in $n^{-1/(4+d)}$. This corresponds to the optimal convergence rate announced in \cite{HeinAudibertvonLuxburg_JMLR}.
}

An important point in the assumptions of Theorem \ref{Theorem: main theorem} is that $K$ is not necessarily continuous nor with mass equal to 1. This can allow to tackle the cases of geometric or kNN graphs for example.\\

This theorem extends the convergence given by Giné and Koltchinskii \cite[Th 5.1]{Gine2006}. They consider the kernel $K(x)=e^{-x^2}$ and control the convergence of the generators uniformly over a class of functions $f$ of class $\cC^3$, uniformly bounded and with uniformly bounded derivatives up to the third order. For such class of functions, the constants in the right hand side (RHS) of \eqref{eq:cv_th} can be made independent of $f$ and we recover a similar uniform bound.\\
The condition \eqref{hyp:h_th11} results from a classical bias-variance trade-off that appears in a similar way in the work of Giné and Koltchinskii \cite{Gine2006}. Notice that the speed $\sqrt{\log h_n^{-1}/(n h_n^{d+2})}$ is also obtained by these authors under the additional assumption that $nh_n^{d+4}/\log h_n^{-1}\rightarrow 0$. We do not make this assumption here. When the additional assumption of Giné and Koltchinskii is satisfied, our rate and their rate coincide as: $h_n^2=o\PAR{\log h_n^{-1}/(nh_n^{d+2})}$.
Hein et al. \cite{HeinAudibertvonLuxburg_COLT2005,HeinAudibertvonLuxburg_JMLR}
extended the results of Giné and Koltchinskii to other kernels $K$, but requesting in particular that these kernels are twice continuously differentiable and with exponential decays (see e.g. \cite[Assumption 2]{HeinAudibertvonLuxburg_COLT2005} or \cite[Assumption 20]{HeinAudibertvonLuxburg_JMLR}). 
Singer \cite{Singer2006}, considering Gaussian kernels, upper bounds the variance term in a
different manner compared to Hein et al., improving their convergence rate when $p$ is the uniform distribution.\\
To our knowledge there are a few works where the consistency of graph Laplacians is proved without continuity assumptions on the kernel $K$. Ting et al. \cite{Ting2010} also worked under the bounded variation assumption on $K$. Additionally, they had to assume that $K$ is compactly supported. In \cite{Calder2022}, Calder and Garcia-Trillos considered a non-increasing kernel with support on $[0,1]$ and Lipschitz continuous on this interval. This choice allows them to consider $K(x)=\ind_{[0,1]}$. Calder and Garc\'ia Trillos established Gaussian concentration of  $\mathcal{A}_{h_n,n}(f)(x)$ and showed that the probability that $|\mathcal{A}_{h_n,n}(f)(x)-\mathcal{A}f(x)|$ exceeds some threshold $\delta$ is exponentially small, of order $\exp(-C\delta^2 nh_n^{d+2})$, when $n\rightarrow +\infty$. In this paper, thanks to the uniform convergence in Theorem \ref{Theorem: main theorem}, we obtain a similar result with additional uniformity on the test functions $f$: 
\begin{corollary}\label{cor:dev_Ahn-A}
Suppose that the density $p$ on the smooth manifold $\cM$ is of class $\cC^2$, and that Assumptions \ref{hyp:K} and \eqref{hyp:h_th11} are satisfied. 
Then there exists a constant $C'>0$ (see \eqref{def:Cprime}), such that for all $n$ and $\delta \in \left[ h_n \vee \sqrt{\frac{\log h_n^{-1}}{nh_n^{d+2}}} ,1 \right]$, we have:
\begin{equation}
\P\PAR{\sup_{f \in \mathcal{F}}\sup_{x \in \mathcal{M}}|\mathcal{A}_{h_n,n}(f)(x)-\mathcal{A}f(x)|>C'\delta}\leq \exp(- nh_n^{d+2} \delta^2),
\end{equation} where $\mathcal{F}$ is the family of $\mathcal{C}^3(\cM)$ functions bounded by 1 and with derivatives up to the third order also bounded by 1. 
\end{corollary}


{The fact that the convergence in Theorem \ref{Theorem: main theorem} is uniform has several other applications. For example, it can be a step to study the spectral convergence for the graph Laplacian using the Courant-Fisher minmax principle
(see e.g. \cite{Calder2022}). 
 Interestingly, the uniform convergence of the Laplacians is also used to study Gaussian free fields on manifolds \cite{Cipriani2020}.\\
}
 
The result of Theorem \ref{Theorem: main theorem} can be extended to the convergence of $\kNN$ Laplacians in the following way. Recall that for $n, k \in \mathbb{N}$ fixed, such that $k\leq n$, the $k$NN graph on the vertices $\{X_1,\dots X_n\}$ is a graph for which the vertices have out-degree $k$. Each vertex has outgoing edges to its $k$-nearest neighbor for the Euclidean distance (again, the geodesic distance could be considered).\\
For $x\in \cM$, the distance between $x$ and its $k$-nearest neighbor is defined as:
\begin{equation}\label{def:dist-kNN}
    R_{n,k}(x)=\inf\Big\{ r\geq 0,\ \sum_{i=1}^n \ind_{\|x-X_i\|_2\leq r}\geq  k\Big\}.
\end{equation}
The Laplacian of the $k$NN-graph is then, for $x \in \mathcal{M}$,
\begin{equation}
\mathcal{A}_{n}^{\kNN}(f)(x):=\frac{1}{nR_{n,k_n}^{d+2}(x)} \sum_{i=1}^n \ind_{[0,1]}\left(\frac{\|X_i-x\|_2}{R_{n,k_n}(x)}\right) (f(X_i)-f(x)).
\end{equation}
A major difficulty here is that the width of the moving window, $R_{n,k_n}(x)$ is random and depends on $x\in \cM$, contrary to the previous $h_n$. The above expression corresponds to the choice of the kernel $K(x)=\ind_{[0,1]}$. 
{The case of $k$NN has been much discussed in the literature but to our knowledge, there are few works where the consistency of $k$NN graph Laplacians have been fully and rigorously considered, because: 1) of the non-regularity of the kernel $K$ and 2) of the fact that the $k$NN graph is {not symmetric,  more precisely,} the vertex $X_i$ is among the $k$-nearest neighbors of a vertex $X_j$ does not imply that $X_j$ is among the $k$-nearest neighbors of $X_i$. {Ting et al. \cite{Ting2010}  discussed that if there is a kind of Taylor expansion with respect to $x$ of the window $R_{n,k_n}(x)$, one might prove a pointwise convergence for $kNN$ graph Laplacian, without convergence rate}. In the present proof, we do not require such Taylor-like expansion. Let us mention also the work of Calder and Garc\'ia Trillos \cite{Calder2022} where the spectral convergence is established for a \textit{symmetrized} version of the $k$NN graph. In the present work, we do not need that local neighborhoods look like balls and deal with the true $k$NN graph. In other papers such as \cite{chengwu}, \eqref{def:dist-kNN} is considered for defining the window width $h_n$ but the kernel $K$ remains continuous. }\\

We will prove the following limit theorem for the rescaled $k$NN Laplacian:
\begin{theorem}\label{main:theorem-kNN}Under Assumption \ref{hyp:K}, if the density {$p\in \cC^2(\cM)$} is such that for all $x\in \cM$,
\begin{equation}
    \label{hyp:p}
    0<p_{\min}\leq p(x)\leq p_{\max},
\end{equation}and if 
    \begin{equation}
    \label{hyp:k_n}
    \lim_{n\rightarrow +\infty}\frac{k_n}{n}=0,\quad \mbox{ and }\quad \lim_{n\rightarrow +\infty} \frac{1}{n}\PAR{\frac{k_n}{n}}^{-1-2/d} \log \PAR{\frac{k_n}{n}} =0,
\end{equation}we have with {probability $1$},
\begin{equation}\label{eq:rate_kNN}
    \sup_{x \in \mathcal{M}} \left|  \mathcal{A}_{n}^{\kNN}(f)(x)- \mathcal{A}(f)(x)\right|= O\left(  \sqrt{  \log \PAR{\frac{n}{k_n}}}\frac{1}{\sqrt{k_n}} \PAR{\frac{n}{k_n}}^{1/d}  {+} \PAR{\frac{k_n}{n}}^{1/d} \right) .
\end{equation}
\end{theorem}
{This theorem is proved in Section \ref{sec:knn}. Notice that the important point in the Assumption \ref{hyp:p} is the lower bound, since in our case of compact manifold, any continuous function $p$ is bounded.
The condition \eqref{hyp:k_n} and the rate of convergence in \eqref{eq:rate_kNN} come from that fact {that the random distance $R_{n,k_n}(x)$ stays with large probability in an interval $[\kappa^{-1}h_n, \kappa h_n]$ for some $\kappa>1$ independent of $x$ and $n$, and for a sequence $h_n$ independent of $x$. This property is based on a result of Cheng and Wu \cite{chengwu}. The proof of Theorem \ref{main:theorem-kNN} follows the main steps presented in the proof of Theorem \ref{Theorem: main theorem} with some slight modifications.} 

Notice that the assumption \eqref{hyp:k_n} is satisfied for 
\[k_n=C n^{1-\alpha},\quad \mbox{ with } \alpha \in \big(0,\frac{1}{d+2}\big),\]for instance.
Optimizing the upper bound in \eqref{eq:rate_kNN} {by varying $\alpha$ in the above choice} gives:
\[k_n= C n^{\frac{4}{d+4}},\]
yielding again a convergence rate of $O\PAR{\sqrt{\log (n)} \ n^{-1/(d+4)}}$}.\\

The rest of the paper is organized as follows. In Section \ref{section:outline}, we give the scheme of the proof. The term $|  \mathcal{A}_{h_n,n}(f)(x)- \mathcal{A}(f)(x)|$ is separated into a bias error, a variance error and a term corresponding to the convergence of the kernel operator to a diffusion operator. In Section \ref{Section: Some geometric background}, we provide some geometric backgrounds that will be useful for the study of the third term, which is treated in Section \ref{Section: Approximation by a sequence of deterministic operators}. The two first statistical terms are considered in Section \ref{Section: approximations by random operators}, which will end the proof of Theorem \ref{Theorem: main theorem}. Corollary \ref{cor:dev_Ahn-A} is then proved at the end of this section. In Section \ref{sec:knn}, we treat the convergence of $\kNN$ Laplacians: after recalling a concentration result for $R_{n,k_n}(x)$, the proof amounts to considering a uniform convergence over a range of window widths.

\begin{notation}
    In this paper $\text{diam}(\mathcal{M})$, $B_{\R^d}(0,r)$ and $ S_{d-1}$ denote  respectively the diameter of $\cM$, $\max_{z,y \in \mathcal{M}}( \| z-y\|_2)$, the ball of $\R^d$ centered at $0$ with radius $r$ and the volume of the  $(d-1)$-unit sphere of $\mathbb{R}^d$.
\end{notation}

\section{Outline of the proof for Theorem \ref{Theorem: main theorem}}\label{section:outline}
First, we focus on the proof of Theorem \ref{Theorem: main theorem}. Recall that $\rho( \cdot,\cdot)$ denotes the geodesic distance on $\mathcal{M}$ and that $\mu(\dd x)$ is the volume measure on $\mathcal{M}$. We define two new operators $\mathcal{A}_{h},$ $\tilde{\mathcal{A}}_{h}$ for each $h>0, x \in \mathcal{M}, f \in \mathcal{C}^3(\mathcal{M})$: 
\begin{align}
	\mathcal{A}_h(f)(x)	&:= \frac{1}{h^{d+2}}\int_{\mathcal{M}} K \left(\frac{ \|x-y\|_2}{h} \right) ( f(y)-f(x))p(y)\mu(\mathrm{d}y) \label{def:Ah}\\
	\tilde{\mathcal{A}}_h(f)(x)	&:= \frac{1}{h^{d+2}}\int_{\mathcal{M}} K \left(\frac{\rho(x,y)}{h} \right) ( f(y)-f(x))p(y)\mu(\mathrm{d}y).\label{def:Atildeh}
\end{align}
The difference between $\mathcal{A}_h$ and $\tilde{\mathcal{A}}_h$ relies in the use of the extrinsic Euclidean distance $\|\cdot\|_2$ for $\mathcal{A}_h$ and of the intrinsic geodesic distance $\rho(\cdot,\cdot)$ for  $\tilde{\mathcal{A}}_h$. Recall here that these two metrics are comparable for close $x$ and $y$:
\begin{theorem}[Approximation inequality for Riemannian distance]\cite[Prop. 2]{Trillos2020} \label{Theorem: Approximation inequality for Riemannian distanceA}There is a constant $c$ such that for $x,y \in \mathcal{M}$, we have: $$\|x-y\|_2 \le \rho(x,y) \le \|x-y\|_2+ c\|x-y\|_2^3.$$
\hfill $\Box$
\end{theorem}

Let us sketch the proof of Theorem \ref{Theorem: main theorem}. By the classical triangular inequality, 
\begin{align}
    \left|  \mathcal{A}_{h_n,n}(f)(x)- \mathcal{A}(f)(x)\right| \leq & 
    \left|   \mathcal{A}(f)(x)-\widetilde{\mathcal{A}}_{h_n}(f)(x)\right|\nonumber\\
    + & \left| \widetilde{\mathcal{A}}_{h_n}(f)(x)- \mathcal{A}_{h_n}(f)(x)\right|\nonumber\\
    + & \left| \mathcal{A}_{h_n}(f)(x)- \mathcal{A}_{h_n,n}(f)(x)\right|
    \label{etape1}
\end{align}

The first term in the RHS of \eqref{etape1} corresponds to the convergence of kernel-based generator to a continuous diffusion generator on $\mathcal{M}$. The following proposition is proved in Section \ref{sec:Proof_Prop_2.4}:
\begin{proposition}[Convergence of averaging  kernel operators]  	 \label{Theorem: convergence of averaging kerel operators}Under Assumption \ref{hyp:K}, {and if $p$ is of class $\cC^2$}. Then, for all $f \in \mathcal{C}^{3}(\mathcal{M})$, we have:
	$$\sup_{x \in \mathcal{M}} \left| \tilde{\mathcal{A}}_h(f)(x)- \mathcal{A}(f)(x) \right| =O(h).$$
\end{proposition}
This approximation is based on tools from differential geometry and exploits the assumed regularities on $K$ and $p$. Similar results have been obtained, in particular by \cite[Th. 3.1]{Gine2006} but with continuous assumptions on $K$ that exclude the kNN cases.\\

The second term in \eqref{etape1} corresponds to the approximation of the Euclidean distance by the geodesic distance and is dealt with the following proposition, proved in Section \ref{sec:proof_Prop_2.3}: 
\begin{proposition}	 \label{Theorem: convergence of averaging kerel operators 2}Under Assumption \ref{hyp:K}, {and for a bounded measurable function $p$}, we have, for all $f$ Lipschitz continuous on $\cM$:
	$$\sup_{x \in \mathcal{M}} \left| \mathcal{A}_h(f)(x)- \mathcal{\tilde{A}}_h(f)(x) \right| =O(h).$$
\end{proposition}

For the last term in the RHS of \eqref{etape1}, note that: 
$$\mathbb{E}\left[ \mathcal{A}_{h_n,n}f(x) \right]=\mathcal{A}_{h_n}f(x),$$
because $(X_i, i \in \mathbb{N})$ are i.i.d. This term corresponds to a statistical error. The following proposition will be proved in Section \ref{Section: approximations by random operators} using Vapnik-Chervonenkis theory:

\begin{proposition}\label{Prop:statistical_error} Under Assumption \ref{hyp:K} and for a bounded measurable function $p$, we have, for all $f \in \cC^{3}(\mathcal{M})$,
$$\sup_{x\in \mathcal{M}} \left|  \mathcal{A}_{h_n,n}f(x) -\mathcal{A}_{h_n}f(x) \right| = O\left(  \sqrt{ \frac{ \log h_n^{-1}}{nh_n^{d+2}}} {+} h_n \right) , \text{a.s.}$$
\end{proposition}

It is worth noticing that there is an interplay between Euclidean and Riemannian distances. On the one hand, the Vapnik-Chervonenkis  theory is extensively studied for Euclidean distances, not for Riemannian distance. On the other hand, approximations on manifolds naturally use local coordinate representations for which the Riemannian distance is well adapted.\\

\section{Some geometric backgrounds}\label{Section: Some geometric background}

\subsection{Riemannian manifold}

Let us recall some facts from differential geometry that will be useful. We refer the reader to \cite{Chavel,Lee2018} for a more rigorous introduction to Riemannian geometry. Let $\mathcal{M}$ be a smooth $d$-dimensional submanifold of $\R^m$.


At each point $x$ of $\mathcal{M}$, there is a tangent vector space $T_x\mathcal{M}$ that contains all the tangent vectors of $\mathcal{M}$ at $x$. The tangent bundle of $\cM$ is denoted by { $T\cM=\sqcup_{x\in \cM}T_x\cM$. For each $x \in \mathcal{M}$, the canonical scalar product $\langle \cdot , \cdot \rangle_{\mathbb{R}^m}$ of $\mathbb{R}^m$ induces a natural scalar product on  $T_x\mathcal{M}$, denoted by $\mathbf{g}(x)$. The application $\mathbf{g}$, which associates each point $x$ with a scalar product on $T_x\cM$, is then called the Riemannian metric on $\cM$ induced by the ambient space $\mathbb{R}^m$.} For $\xi,\eta\in T_x\cM$, we use the classical notation { $\langle \xi,\eta\rangle_{\mathbf{g}}$ to denote the scalar product of $\xi$ and $\eta$ w.r.t to the scalar product $\mathbf{g}(x)$.  }\\

Consider a coordinate chart $\Phi=(x^1,\dots x^d): U \rightarrow \mathbb{R}^d$ on a neighborhood $U$ of $x$. 
Denoting by $\left\{ \left.\frac{\partial}{\partial x^1} \right|_x, \left.\frac{\partial}{\partial x^2} \right|_x,..., \left.\frac{\partial}{\partial x^d} \right|_x\right\}$ the natural basis of $T_x\mathcal{M}$ associated with the coordinates $(x^1,\dots x^d)$. Then, the scalar product $\mathbf{g}(x)$ is associated to a matrix $(g_{ij})_{i,j\in \lbrac 1, d\rbrac}$ {in the sense that}
in this coordinate chart, for $\xi$ and $\eta\in T_x\cM$,
\begin{equation}\langle \xi,\eta\rangle_{\mathbf{g}} = \sum_{i,j=1}^d g_{ij}(x) \xi^i\eta^j,\end{equation}
{where $(\xi^i),(\eta^j) $ are the coordinates of $\xi$ and $\eta$ in the above basis of $T_x\cM$.}
Notice that, for each $i,j \in [\![1,d]\!]$
\begin{equation}
	g_{ij}(x):=\left\langle \left.\frac{\partial}{\partial x^i} \right|_x, \left.\frac{\partial}{\partial x^j} \right|_x \right\rangle_{\mathbf{g}},
\end{equation}
and $g_{ij}:U \subset M\rightarrow \mathbb{R}$ is smooth. For a real function $f$ on $\cM$, we will denote
$\widehat{f}$ its expression in the local chart: $\widehat{f}=f\circ \Phi^{-1}$. Recall that the derivative $\frac{\partial f}{\partial x^j}$ is defined as: 
\[\frac{\partial f}{\partial x^j}:=\frac{\partial \widehat{f}}{\partial x^j}\circ \Phi.\]
Also we denote
\begin{equation}
    \widehat{g}_{ij}:= g_{ij}\circ \Phi^{-1},
\end{equation}
which will be called the coordinate representation of the Riemannian metric in the local chart $\Phi$.  \\

\par Charts (from $U\subset \cM\to \R^d$)  induce local parameterizations of the manifold (from $\R^d \to U\subset \cM$). Among all possible local coordinate systems of a neighborhood of $x$ in $\cM$, there are \textit{normal coordinate charts} (see \cite[p. 131-132]{Lee2018} or the remark below for a definition). We denote by $\mathcal{E}_x$ the Riemannian normal parameterization at $x$, i.e., $\mathcal{E}^{-1}_x$ is the corresponding normal coordinate chart.

\begin{remark}[Construction of $\mathcal{E}_x^{-1}$]\label{rque:normal_chart_param}For the sake of completeness, we briefly recall the construction of \cite{Lee2018}.
Let $U$ be an open subset of $\cM$. There exists a local orthonormal frame $(E_i)_{i\in \lbrac 1, d\rbrac}$  over $U$, see \cite[Prop. 2.8, p. 14]{Lee2018}.
The tangent bundle $TU$ can be identified with $U\times \R^d$ thanks to the smooth map:
\begin{equation}\label{eq:def_F}F \ : \ \begin{array}{ccc}
U\times \R^d & \rightarrow  & TU \\
(x,(v_1,\dots v_d)) & \mapsto  & v=\sum_{i=1}^d v_i E_i|_x.
\end{array}\end{equation}
 So for each $x\in U$, $F(x,\cdot)$ is an isometry between $\R^d$ and $T_x\cM$. \\
 Recall that by \cite[Prop 5.19, p. 128]{Lee2018}, the exponential map $\exp(\cdot)$ of $\cM$ can be defined on a non-empty open subset $W$ of $T\cM$ such that $\forall x \in \cM$, $\overrightarrow{0}_x \in W $, where $\overrightarrow{0}_x$ is the zero element of $T_x\cM$. Then, the map $\exp\circ F : (x,v) \mapsto \mathcal{E}_x(v):=\exp\circ\, F (x,(v_1,\dots, v_d))$ is well-defined on $F^{-1}(W \cap TU )$ and $\mathcal{E}_x^{-1}$ is a Riemannian normal coordinate chart at $x\in U$, smooth with respect to $x$.
 \end{remark}

Let us state some properties of the normal coordinate charts.
 
\begin{theorem}[Derivatives of Riemannian metrics in normal coordinate charts]\label{Theorem: derivatives of Riemannian metrics in normal coordinate chartsA}\cite[Prop. 5.24]{Lee2018} For $x\in \cM$, let $\mathcal{E}_x^{-1} : U\subset \mathcal{M}  \rightarrow \mathbb{R}^d $ be a normal coordinate chart at a point $x$ such that $\mathcal{E}_x^{-1}(x)=0$ and let $( \widehat{g}_{ij}; 1 \le i, j \le d)$ be the coordinate representation of the Riemannian metric of $\mathcal{M}$ in the local chart $\mathcal{E}_x^{-1}$. Then for all $i,j$, 
	\begin{equation}
		\widehat{g}_{ij}(0)=\delta_{ij}, \quad \widehat{g}'_{ij}(0)=0, 
	\end{equation}where $\delta_{ij}$ is the Kronecker delta.
Additionally, for all $y \in U$,
\begin{equation}
\rho(x,y)= \|  \mathcal{E}_x^{-1}(y)\|_2.\label{eq:rho=normePhi}
\end{equation}
\end{theorem}

\begin{notation} For any function $f: \mathbb{R}^{d} \rightarrow \mathbb{R}^{k}$, we denote by $f' : \mathbb{R}^{d} \rightarrow \mathbb{R}^{k}$ the linear map that represents the first order derivative of $f$. Similarly, we denote respectively by $f'' : \mathbb{R}^{d}\times \mathbb{R}^d \rightarrow \mathbb{R}^{k}$ and $f''': \mathbb{R}^{d}\times\mathbb{R}^{d}\times \mathbb{R}^{d} \rightarrow \mathbb{R}^{k}$ the bi-linear map and the tri-linear map that represent the second order derivative and the third order derivative of $f$. Thus, the Taylor's expansion of $f$ up to third order can be written as $$f(x+v)= f(x) +f'(x)(v)+\frac{1}{2}f''(x)(v,v)+\frac{1}{6}f'''(x+\varepsilon v)(v,v,v),$$
for some $\varepsilon \in (0,1)$.
\end{notation}

For the normal parameterizations $\mathcal{E}_x$, we now state some uniform controls that are keys for our computations in the sequel.

\begin{theorem}[Existence of a "good" family of parameterizations.] \label{Theorem: Existence of a "good" family of normal coordinate systemsA}
	There exist constants $c_1,c_2>0$ and a family $(\mathcal{E}_x, x \in \mathcal{M})$ of smooth local parameterizations of $\mathcal{M}$ which have the same domain $B_{ \mathbb{R}^d}(0,c_1)$ such that for all $x \in \mathcal{M}$,
	\begin{itemize}
		\item[i.]  $\mathcal{E}_x^{-1}$ is a normal coordinate chart of $\mathcal{M}$ and $\mathcal{E}_x(0)=x$.
		\item[ii.] For $v\in B_{ \mathbb{R}^d}(0,c_1)$, we denote by $( \widehat{g}^{x}_{ij}(v); 1 \le i, j \le d)$ the coordinate representation of the Riemannian metric $\bg(\mathcal{E}_x(v))$ of $\mathcal{M}$ in the local parameterization $\mathcal{E}_x$. Then for all $v\in  B_{ \mathbb{R}^d}(0,c_1)$:
		\begin{equation}\label{eq:maj-detg}
				\left| \sqrt{\mathrm{det}\, \widehat{g}^{x}_{ij} (v)} -1 \right| \le c_2 \| v\|_2^2.
		\end{equation}
		\item[iii.] We have $\| \mathcal{E}_x(v)- x\|_2 \le \|v\|_2$. In addition, for all $v \in B_{\mathbb{R}^d}(0,c_1)$,

		\begin{equation}\label{eq:estimeeEronde1}
	\| \mathcal{E}_x(v)- x- \mathcal{E}_x'(0)(v)\|_2 \le c_2\|v\|_2^2,
\end{equation}and 
\begin{equation}\label{eq:estimeeEronde2}
			\| \mathcal{E}_x(v)- x- \mathcal{E}_x'(0)(v)-  \frac{1}{2}\mathcal{E}_x''(0)(v,v)\|_2 \le c_2\|v\|_2^3,
		\end{equation}
	\end{itemize}
\end{theorem}
\begin{proof}[Proof for Theorem \ref{Theorem: Existence of a "good" family of normal coordinate systemsA}]
Let $U$ be an open domain of a local chart of $\cM$. Following Remark \ref{rque:normal_chart_param} and noticing that there is always an orthonormal frame over $U$, we can define a family of normal parameterizations $(\mathcal{E}_x)_{x \in U}$.

First, we easily note that $\| \mathcal{E}_x(v)- x\|_2 \le \|v\|_2$ thanks to Theorem \ref{Theorem: Approximation inequality for Riemannian distanceA} and $\eqref{eq:rho=normePhi}$.\\
Restricting $U$ if necessary (by an open subset with compact closure in $U$), there exists constants $c_1,c_2>0$ such that $\exp\circ F$ is well-defined on $U\times B_{\mathbb{R}^d}(0,c_1)$ and that all $v\in B_{\R^d}(0,c_1)$, \eqref{eq:estimeeEronde1}-\eqref{eq:estimeeEronde2} hold by Taylor expansions of $\mathcal{E}_x$. Equation \eqref{eq:maj-detg} is a consequence of the smoothness of $\mathcal{E}_x$ and Theorem \ref{Theorem: Existence of a "good" family of normal coordinate systemsA}.

Clearly, for each point $y\in \cM$, we can find an open neighborhood $U$ of $y$ and positive constants $c_1$ and $c_2$ such as above.  Hence, such open sets form an open covering of $\cM$. Therefore, by the compactness of $\cM$, there exists a finite covering of $\cM$ by such open sets $U$ and therefore, the constants $c_1$ and $c_2$ can be chosen uniformly for all $\mathcal{E}_x$.	
\end{proof}

\subsection{Gradient operator, Laplace-Beltrami operator}

Given a Riemannian manifold $(\mathcal{M},\mathbf{g})$, the gradient operator $\nabla_{\mathcal{M}}$ and the Laplace-Beltrami operator $\Delta_{\mathcal{M}}$ are,  as suggested by their names, the generalizations for differential manifolds of the gradient $\nabla_{\mathbb{R}^m}$, the Laplacian $\Delta_{\mathbb{R}^m}$ in the Euclidean space $\mathbb{R}^m$.

For a function $f$ of class $\cC^1$ on $\cM$, the gradient $\nabla_\cM f$ is expressed in local coordinates as
\begin{equation}\label{eq:def-gradient}
\nabla_\cM f(x)=\sum_{i,j=1}^dg^{ij}(x)\frac{\partial f}{ \partial x^i}(x)\left.\frac{\partial}{\partial x^j}\right\vert_x,
\end{equation}
where $\PAR{g^{ij}}_{1\leq i,j,\leq d}$ is the inverse matrix of $\PAR{g_{ij}}_{1\leq i,j,\leq d}$. Since $\sum_{j=1}^dg^{ij}g_{jk}=\delta_{ik}$, we note that for $f,h$ functions of class $\cC^1$,
\begin{equation}\label{eq:grad}
\SCA{\nabla_\cM( f),\nabla_\cM (h)}_\bg=\sum_{i,j=1}^dg^{ij}\frac{\partial f}{\partial x^i}\frac{\partial h}{\partial x^j}.
\end{equation}

The Laplace-Beltrami operator is defined by (see \cite[Section 3.1]{Hsu2002})
\begin{equation}\label{eq:def-LapB}
\lapM f:=\sum_{i,j=1}^d \frac{1}{\sqrt{\det(\bg)}} \frac{\partial}{ \partial x^i}\PAR{\sqrt{\det(\bg)}g^{ij}\frac{\partial f}{ \partial x^j}}.
\end{equation}

When using normal coordinates, the expressions of the Laplacian and the gradient of a smooth function $f$ at a point $x$ match their definitions in $\R^d$.
\begin{proposition}\label{Proposition: some calculations involving Laplacians and gradientsA}
	Suppose that $\Phi:  U\subset \mathcal{M}  \rightarrow \mathbb{R}^d$ is a normal coordinate chart at a point $x$ in $\mathcal{M}$ such that $\Phi(x)=0$, then:
	\begin{itemize}
\item[i.]   $\langle \nabla_{\mathcal{M}} f (x) ,\nabla_{\mathcal{M}} h(x) \rangle_\bg =\langle \nabla_{\mathbb{R}^d}  \hat{f} (0) ,\nabla_{\mathbb{R}^d} \hat{h}(0) \rangle.  $
\item[ii.] $ \Delta_{\mathcal{M}} f(x) = \Delta_{\mathbb{R}^d}\hat{f}(0),  $
	\end{itemize}
\end{proposition} 
\begin{proof}[Proof for Proposition \ref{Proposition: some calculations involving Laplacians and gradientsA}]
Recall that $g_{ij}(x)=\widehat{g}_{ij}(0)$. By Theorem  \ref{Theorem: derivatives of Riemannian metrics in normal coordinate chartsA}, we know that $\hat{g}_{ij}(0)= \delta_{ij}$, thus, $\widehat{g}^{ij}(0)=\delta_{ij}$ and \textit{i.} is a consequence of \eqref{eq:grad}. For the equality \textit{ii.}, we use \eqref{eq:def-LapB}. Since for the normal coordinates $\text{det}\,\widehat{\bg} (0)=1$ and since the derivatives of $\widehat{g}_{ij}$ and $\widehat{g}^{ij}$ vanish at $0$, we have the conclusion.
\end{proof}



\section{Some kernel-based approximations of $\mathcal{A}$}
\label{Section: Approximation by a sequence of deterministic operators}

The aim of this Section is to prove the estimates for the two error terms in the RHS of \eqref{etape1} and prove the Propositions \ref{Theorem: convergence of averaging kerel operators} and \ref{Theorem: convergence of averaging kerel operators 2}. Both error terms are linked with the geometry of the problem and use the results presented in Section \ref{Section: Some geometric background}. The first one deals with the approximation of the Laplace-Beltramy operator by a kernel estimator (see Section \ref{sec:Proof_Prop_2.4}), while the second one treats the differences between the use of the Euclidean norm of $\R^m$ and the use of the geodesic distance (see Section \ref{sec:proof_Prop_2.3}).




\subsection{Weighted moment estimates}

We begin with an auxiliary estimation. The result is related to kernel smoothing and can also be useful in density estimation on manifolds (see e.g. \cite{berenfeldhoffman}). 
\begin{lemma} \label{Lemma: An useful estimation}
 Under Assumption \ref{hyp:K}, uniformly in $x \in \mathcal{M}$, when $h $ converges to $0$, we have:
 \begin{align}
 	& \frac{1}{h^{d+2}}\int_{\mathcal{M}} \mathbb{1}_{ \rho(x,y) \ge c_1 } K \left(\frac{\rho(x,y)}{h} \right)  \mu(\mathrm{d}y)=o(h),\label{Equation: rho>c_1 inequality}\\
 	& \frac{1}{h^{d+2}}\int_{\mathcal{M}} \mathbb{1}_{ \rho(x,y) \ge c_1 } K \left(\frac{\|x-y\|_2 }{h} \right)  \mu(\mathrm{d}y)=o(h),\label{Equation: rho>c_1 inequality_bis}
 \end{align}
 and there is a generic constant $c$ such that for all point $x \in \mathcal{M}$ and positive number $h>0$, we have:
 \begin{eqnarray}
 			\frac{1}{h^{d+2}}\int_{\mathcal{M}} K \left(\frac{\rho(x,y)}{h} \right) \|x -y\|_2^3 \mu(\mathrm{d}y) & \le&  ch ,
 	\label{Equation: an useful estimation 1}
 	\\
 		\frac{1}{h^{d+2}}\int_{\mathcal{M}} K \left(\frac{\rho(x,y)}{h} \right) \|x -y\|_2^2 \mu(\mathrm{d}y) & \le& c,
 	\label{Equation: an useful estimation 2}
 	\\
 		\frac{1}{h^{d+2}}\int_{\mathcal{M}} K \left(\frac{\|x-y\|_2}{h} \right) \|x -y\|_2^3 \mu(\mathrm{d}y) & \le& ch ,
 \label{Equation: an useful estimation 3}
 \\
 \frac{1}{h^{d+2}}\int_{\mathcal{M}} K \left(\frac{\|x-y\|_2}{h} \right) \|x -y\|_2^2 \mu(\mathrm{d}y)& \le & c.
 \label{Equation: an useful estimation 4}
 \end{eqnarray}

\end{lemma}	

\begin{proof}[Proof of Lemma \ref{Lemma: An useful estimation}]
Using Lemma \ref{Lemma: an inequality on bounded variation functions}, we have:
\begin{align}
		\int_{\mathcal{M}} \mathbb{1}_{ \rho(x,y) \ge  c_1 } K \left(\frac{\rho(x,y)}{h} \right)  \mu(\mathrm{d}y)  & \le \mu(\mathcal{M}) \sup_{r \ge c_1} K\left( \frac{r}{h}\right)  \nonumber \\
		 & \le \mu(\mathcal{M}) \left[ H(\infty)-H\left( \frac{c_1}{h}\right) \right]   \nonumber\\
		 & = \mu(\mathcal{M})  \int_{(c_1/h,\infty)} \textrm{d}H(a) 
		 \nonumber \\
		 &\le h^{d+3} \frac{ \mu(\mathcal{M})}{c_1^{d+3}} \int_{(c_1/h,\infty)} a^{d+3}\textrm{d}H(a).  
\end{align}
Thanks to the boundedness of  $ \int_0^{\infty}a^{d+3}\mathrm{d}H(a)$, we obtain \eqref{Equation: rho>c_1 inequality}.
Then,  as a consequence of \eqref{Equation: rho>c_1 inequality}, by the compactness of $\mathcal{M}$, we easily observe that uniformly in $x \in \mathcal{M}$, when $h$ converges to $0$, $$\frac{1}{h^{d+2}}\int_{\mathcal{M}} \mathbb{1}_{ \rho(x,y) \ge c_1 } K \left(\frac{\rho(x,y)}{h} \right) \ \|x -y\|_2^3 \ \mu(\mathrm{d}y)=o(h). $$ \\
So, to prove Inequality \eqref{Equation: an useful estimation 1}, it is left to prove that uniformly in $x$,  when $h$ converges to $0$,
\begin{equation}\label{Equation: a partial integral}
	I:=\frac{1}{h^{d+3}}\int_{\mathcal{M}} \mathbb{1}_{ \rho(x,y) < c_1 } K \left(\frac{\rho(x,y)}{h} \right) \|x -y\|_2^3 \mu(\mathrm{d}y)=O(1). 
\end{equation}
Recall that in Theorem \ref{Theorem: Existence of a "good" family of normal coordinate systemsA}, we showed that for each point $x \in\mathcal{M}$, there is a local smooth parameterization $\mathcal{E}_x$ of $\mathcal{M}$ that has many nice properties, especially $\rho(x,y)= \| \mathcal{E}_x^{-1}(y)\|_2$ for all $y$ within an appropriate neighborhood of $x$ by \eqref{eq:rho=normePhi}. Thus, the term $I$ in the left hand side (LHS) of \eqref{Equation: a partial integral} can be re-written in its coordinate representation under the parameterization $\mathcal{E}_x$ by using the change of variables $v=\mathcal{E}_x^{-1}(y)$:
\begin{align}
			I &=\frac{1}{h^{d+3}} \int_{ B_{\mathbb{R}^d}(0,c_1)}K \left(\frac{\|v\|_2}{h} \right) \| x -\mathcal{E}_x(v)\|_2^3 \sqrt{\text{det}\widehat{g}^{x}_{ij}}(v)\mathrm{d}v.  \nonumber
\end{align}
Then, using Theorem \ref{Theorem: Approximation inequality for Riemannian distanceA}  and Theorem \ref{Theorem: Existence of a "good" family of normal coordinate systemsA} ($ii$ and $iii$),
\begin{align}
I			&\le \frac{c_2}{h^{d+3}}\int_{ B_{\mathbb{R}^d}(0,c_1)}K \left(\frac{\|v\|_2}{h} \right) \| v \|_2^3 (1+ c_2\|v\|_2^2)\text{d}v  \nonumber \\
			& \le \frac{c_2}{h^{d+3}}\int_{ B_{\mathbb{R}^d}(0,c_1)}K \left(\frac{\|v\|_2}{h} \right) \| v \|_2^3 (1+ c_2c_1^2)\text{d}v  \nonumber \\
			& \le \frac{c_2}{h^{d+3}}\int_{ \mathbb{R}^d}K \left(\frac{\|v\|_2}{h} \right) \| v \|_2^3 (1+ c_2c_1^2)\text{d}v. \nonumber
			\\
			&= c_2(1+c_2c_1^2)\int_{\mathbb{R}^d} K( \|v\|_2)\|v\|_2^3 \text{d}v, \label{Equation: a partial integral 2}
\end{align}
Using the spherical coordinate system when $d\geq 2$: 
\begin{align}
I & \leq c_2(1+c_2c_1^2) \left[  \int_{0}^{\infty} K(a) a^3 \times a^{d-1}\text{d}a \right] \times \nonumber \\
				& \qquad \times\left[  \int_{ [0,2\pi]\times[0,\pi]^{d-2}}  \sin^{d-2}(\theta_1)\sin^{d-3}(\theta_2) \cdots \sin(\theta_{d-2})\mathrm{d} \theta\right]\nonumber
				\\
				&=  c_2(1+c_2c_1^2)\ S_{d-1} \left[  \int_{0}^{\infty} K(a) a^{d+2}\text{d}a \right]. \nonumber
\end{align}For $d=1$, we use that $\int_{\mathbb{R}^1} K(|v|)|v|^3 \mathrm{d}v=2\times \int_{0}^{\infty} K(a)a^{1+2}\mathrm{d}a$.
\\
Hence, by Lemma \ref{Lemma: an inequality on bounded variation functions} in the Appendix, and  Fubini's theorem, we have:
\begin{align}
		I	&\le  c_2(1+c_2c_1^2) S_{d-1}	 \int_{0}^{\infty} \left( H(\infty)-H(a) \right) a^{d+2}\text{d}a  \nonumber
	\\
	&=  c_2(1+c_2c_1^2)(d+3)^{-1}S_{d-1}  \int_{[0,\infty]} b^{d+3} \textrm{d}H(b)<\infty  \label{Equation: good inequality 2}.
\end{align}
Therefore, Inequality \eqref{Equation: an useful estimation 1} is proved. The proof of Inequality \eqref{Equation: an useful estimation 2} is similar.\\
For Inequalities \eqref{Equation: rho>c_1 inequality_bis}, \eqref{Equation: an useful estimation 3} and  \eqref{Equation: an useful estimation 4}, we observe that they are indeed consequences of \eqref{Equation: rho>c_1 inequality}, \eqref{Equation: an useful estimation 1}  and \eqref{Equation: an useful estimation 2}. Consider for example \eqref{Equation: rho>c_1 inequality_bis}, using again Lemma \ref{Lemma: an inequality on bounded variation functions} and Theorem \ref{Theorem: Approximation inequality for Riemannian distanceA}:
\begin{align}
		&\frac{1}{h^{d+2}}\int_{ \mathcal{M}} \ind_{\rho(x,y)\geq c_1}K \left(\frac{\|x-y\|_2}{h} \right)  \mu(\mathrm{d}y) \nonumber \\
	\le &\frac{1}{h^{d+2}}\int_{ \mathcal{M}}  \ind_{\rho(x,y)\geq c_1}\left[ H(\infty)- H \left( \frac{\rho(x,y)}{h(1+c_3\|x-y\|_2^2)}  \right)  \right] \mu(\mathrm{d}y) \nonumber
	\\
	\le &\frac{1}{h^{d+2}}\int_{ \mathcal{M}} \ind_{\rho(x,y)\geq c_1} \left[ H(\infty)- H \left( \frac{\rho(x,y)}{h(1+c_3\text{diam}(\mathcal{M})^2)}  \right)  \right] \mu(\mathrm{d}y) \nonumber
	\\
	= &\frac{1}{h^{d+2}}\int_{ \mathcal{M}}\ind_{\rho(x,y)\geq c_1}   \tilde{K} \left(\frac{\rho(x,y)}{h} \right)  \mu(\mathrm{d}y), \label{Inequality: an argument to replace K}
	\end{align}
for $\tilde{K} (a) := H(\infty)- H\left(   \frac{a}{1+c_3\text{diam}(\mathcal{M})^2}\right) $ and where the second inequality uses that $H$ is a non-decreasing function. So Inequality \eqref{Equation: rho>c_1 inequality_bis} corresponds to Inequality \ref{Equation: rho>c_1 inequality} where $K$ is replaced with $\tilde{K}$. Clearly, the function $\tilde{K}$ is of bounded variation and satisfies Assumption \ref{hyp:K}, which conclude the proof for \eqref{Equation: rho>c_1 inequality_bis}. The arguments are similar for \eqref{Equation: an useful estimation 3} and  \eqref{Equation: an useful estimation 4}.
\end{proof}

\subsection{Proof of Proposition \ref{Theorem: convergence of averaging kerel operators}}\label{sec:Proof_Prop_2.4}

In this section, we prove Proposition \ref{Theorem: convergence of averaging kerel operators}, dealing with the approximation of the Laplace Beltrami operator by a kernel operator.\\

In the course of the proof, some quantities involving gradients and Laplacian will appear repetitively. The next lemma deal will be useful to deal with these expressions and its proof is postponed to Appendix \ref{app:proofLemma_aux}:
\begin{lemma}(Some auxiliary calculations)
	\label{Lemma: Some auxiliary calculations.}
	Suppose that $f,h: \mathbb{R}^m \rightarrow \mathbb{R}$, $k: \mathbb{R}^d \rightarrow \mathbb{R}^m$ are $\mathcal{C}^2$-continuous functions, that $k(0)=x$ and suppose that $G: \mathbb{R}_+  \rightarrow \mathbb{R}$ is a locally bounded measurable function. Then, for all $c>0$:
	\begin{multline}
		\int_{B_{\mathbb{R}^d}(0,c)}  G(\|v\|_2) \langle \nabla_{\mathbb{R}^m}f(x) , k'(0)(v)\rangle \langle \nabla_{\mathbb{R}^m}h(x) , k'(0)(v)\rangle \ \mathrm{d}v \\
		=\langle \nabla_{\mathbb{R}^d}( f\circ k )(0), \nabla_{\mathbb{R}^d}( h\circ k )(0)\rangle \left( \frac{1}{d}\int_{B_{\mathbb{R}^d}(0,c)}  G(\|v\|_2)  \|v\|_2^2 \mathrm{d}v \right), \label{Equation: first equation in the auxiliary lemma}
	\end{multline}
and that:
\begin{multline}
	\label{Equation: second equation in the auxiliary lemma}
	\int_{B_{\mathbb{R}^d}(0,c)} G \left(\|v\|_2 \right) \bigg[ \left\langle \nabla_{\mathbb{R}^m} f(x),k'(0)(v)+\frac{1}{2}k''(0)(v,v) \right\rangle 
	\\ + \frac{1}{2}f''(x)(k'(0)(v),k'(0)(v))  \bigg] \mathrm{d}v\\
	=
 \frac{1}{2}\Delta_{\mathbb{R}^d}(f\circ k)(0)\left( \frac{1}{d}\int_{B_{\mathbb{R}^d}(0,c)}  G(\|v\|_2)  \|v\|_2^2 \mathrm{d}v \right). 
\end{multline}
\end{lemma}

Let us now prove Proposition \ref{Theorem: convergence of averaging kerel operators}.
As $\mathcal{M}$ is compact, $\mathcal{M}$ is a properly embedded  submanifold of $\mathbb{R}^m$ (see \cite[p.98]{Lee2013}). Hence, any  function of class $\mathcal{C}^3$ on $\mathcal{M}$ can be extended to a function of class $\mathcal{C}^3$  on $\mathbb{R}^m$, see \cite[Lemma 5.34]{Lee2013}. So without loss of generality, assume $f$ and $p$ are $\mathcal{C}^3$ and $\mathcal{C}^2$ functions on $\mathbb{R}^m$ with compact supports.\\
Recall that we want to study $\ABS{\widetilde{\cA}_h(f)-\cA(f)}$ where $\mathcal{A}$ and $\widetilde{\mathcal{A}}_h$ have been respectively defined in \eqref{Equation: differential operator} and \eqref{def:Atildeh}.
So, introducing the constant $c_1>0$ of Lemma \ref{Lemma: An useful estimation}  and noticing that  $f$ and $p$ are uniformly bounded on the compact $\mathcal{M}$, to prove Proposition \ref{Theorem: convergence of averaging kerel operators}, we only have to prove that uniformly in $x\in \cM$,
$$\left|\mathcal{A}(f)(x)- \frac{1}{h^{d+2}}\int_{\mathcal{M}} \mathbb{1}_{\rho(y,x)<c_1}K \left(\frac{\rho(x,y)}{h} \right) ( f(y)-f(x))p(y)\mu(\mathrm{d}y) \right|=O(h).$$
Besides, thanks to the compactness of $\mathcal{M}$ and to the regularity of $f$ and $p$, Taylor's expansion implies that there is a constant $c_4$ such that for all $x,y \in \mathcal{M}$:
\begin{multline}
	\Bigg| (f(y) -f(x))p(y) \\
	-\bigg( \langle \nabla_{\mathbb{R}^m} f(x),y-x \rangle +\frac{1}{2}f''(x)(y-x,y-x) \bigg)p(x)  \\ 
	- \langle \nabla_{\mathbb{R}^m}  f(x),y-x \rangle \langle \nabla_{\mathbb{R}^m}  p(x),y-x \rangle\Bigg| \le c_4 \|x-y\|_2^3. 
\end{multline} 
Hence, by Inequality \eqref{Equation: an useful estimation 1}, it is sufficient to prove that uniformly in $x$, 
\begin{align} \label{Equation: an integral that we have to change}
	J_1:= & \frac{1}{h^{d+2}}\int_{\mathcal{M}}\mathbb{1}_{\rho(y,x)<c_1} K \left(\frac{\rho(x,y)}{h} \right) \langle \nabla_{\mathbb{R}^m} f(x),y-x \rangle \times \nonumber\\
	& \hspace{6cm} \langle \nabla_{\mathbb{R}^m} p(x),y-x \rangle\mu(\mathrm{d}	y) \nonumber\\
	= & c_0\SCA{\nabla_{\mathcal{M}}(f)(x), \nabla_{\mathcal{M}}(p)(x)}_\bg+O(h) .
\end{align}
and
\begin{align}
	J_2:=& \frac{1}{h^{d+2}}\int_{\mathcal{M}} \mathbb{1}_{\rho(y,x)<c_1}K \left(\frac{\rho(x,y)}{h} \right) \bigg[ \langle \nabla_{\mathbb{R}^m} f(x),y-x \rangle \nonumber\\
	& \hspace{6cm} +\frac{1}{2}f''(x)(y-x,y-x) \bigg] \mu(\mathrm{d}y) \nonumber\\
	=&  \frac{1}{2}c_0\Delta_{\mathcal{M}}(f)(x)+O(h).  \label{Equation: an estimation has to be done}
\end{align}
The proof is similar than the study of $I$ given by \eqref{Equation: a partial integral} in the proof of Lemma \ref{Lemma: An useful estimation}.
We re-write the considered integrals in coordinate representations. Using the change of variables $v=\mathcal{E}_x^{-1}(y)$, we have
\begin{multline*}
J_1=\frac{1}{h^{d+2}}\int_{B_{\mathbb{R}^d}(0,c_1)} K \left(\frac{\|v\|_2}{h} \right) \langle \nabla_{\mathbb{R}^m} f(x),\mathcal{E}_x(v)-x \rangle\\
\times\langle \nabla_{\mathbb{R}^m} p(x),\mathcal{E}_x(v)-x \rangle \sqrt{ \mathrm{det}\widehat{g}^x_{ij}(v)} \mathrm{d}v.
\end{multline*}
By properties $ii.$ and $iii.$ in Theorem \ref{Theorem: Existence of a "good" family of normal coordinate systemsA} we have
\begin{align}
	\Bigg|J_1
	-	&\int_{B_{\mathbb{R}^d}(0,c_1)} K \left(\frac{\|v\|_2}{h} \right) \langle \nabla_{\mathbb{R}^m} f(x),\mathcal{E}_x(v)-x \rangle \langle \nabla_{\mathbb{R}^m} p(x),\mathcal{E}_x(v)-x \rangle  \mathrm{d}v\Bigg| \nonumber
	\\
		\le &\frac{c_2}{h^{d+2}} \|\nabla_{\mathbb{R}^m} f(x)\|_2 \| \nabla_{\mathbb{R}^m} p(x)\|_2 \int_{B_{\mathbb{R}^d}(0,c_1)} K \left(\frac{\|v\|_2}{h} \right)    \|v\|_2^2 .\| \mathcal{E}_x(v)-x\|_2^2 \mathrm{d}v \nonumber
	\\
	\le &\frac{c_2^3}{h^{d+2}} \|\nabla_{\mathbb{R}^m} f(x)\|_2 \| \nabla_{\mathbb{R}^m} p(x)\|_2 \int_{B_{\mathbb{R}^d}(0,c_1)} K \left(\frac{\|v\|_2}{h} \right)    \|v\|_2^4  \mathrm{d}v \nonumber
	\\
	\le & \frac{c_2^3c_1}{h^{d+2}} \|\nabla_{\mathbb{R}^m} f(x)\|_2 \| \nabla_{\mathbb{R}^m} p(x)\|_2 \int_{B_{\mathbb{R}^d}(0,c_1)} K \left(\frac{\|v\|_2}{h} \right)    \|v\|_2^3  \mathrm{d}v. \nonumber
\end{align}
As in the proof of Lemma \ref{Lemma: An useful estimation}, we deduce that the latter is bounded by $O(h)$.
\\
Besides, using again Property $iii.$ in Theorem \ref{Theorem: Existence of a "good" family of normal coordinate systemsA}, we have that uniformly in $x$,
\begin{multline}
\Big|\frac{1}{h^{d+2}}\int_{B_{\mathbb{R}^d}(0,c_1)} K \left(\frac{\|v\|_2}{h} \right) \langle \nabla_{\mathbb{R}^m} f(x),\mathcal{E}_x(v)-x \rangle \times \\
\langle \nabla_{\mathbb{R}^m} p(x),\mathcal{E}_x(v)-x \rangle  \mathrm{d}v
-J_{11}\Big|
=O(h)
\end{multline}with
\begin{align*}
J_{11}:= &\frac{1}{h^{d+2}}\int_{B_{\mathbb{R}^d}(0,c_1)} K \left(\frac{\|v\|_2}{h} \right) \langle \nabla_{\mathbb{R}^m} f(x),\mathcal{E}_x'(0)(v) \rangle \langle \nabla_{\mathbb{R}^m} p(x),\mathcal{E}_x'(0)(v) \rangle  \mathrm{d}v.
\end{align*}
Let us now compare $J_{11}$ to the first term of the generator $\cA$.
Using Equation \eqref{Equation: first equation in the auxiliary lemma} of Lemma \ref{Lemma: Some auxiliary calculations.}, with $G(||v||_2)=K\PAR{\frac{||v||_2}{h}}$, $k=\mathcal{E}_x$, and Proposition \ref{Proposition: some calculations involving Laplacians and gradientsA}, we have:
\begin{align*}
J_{11} =&\frac{1}{h^{d+2}}\left( \frac{1}{d}\int_{B_{\mathbb{R}^d}(0,c_1)}  K\left(\frac{\|v\|_2}{h}\right)  \|v\|_2^2 \mathrm{d}v \right) \langle \nabla_{\mathbb{R}^d}( f\circ \mathcal{E}_x )(0), \nabla_{\mathbb{R}^d}( p \circ \mathcal{E}_x )(0)\rangle 
\\
=& \left( \frac{1}{d}\int_{B_{\mathbb{R}^d}(0,c_1/h)}  K\left(\|v\|_2\right)   \|v\|_2^2 \mathrm{d}v \right) \langle \nabla_{\mathcal{M}} f(x), \nabla_{\mathcal{M}} p(x) \rangle_\bg
\\
=& \left( \frac{1}{d}\int_{\mathbb{R}^d}  K\left(\|v\|_2\right)   \|v\|_2^2 \mathrm{d}v \right) \langle \nabla_{\mathcal{M}} f(x), \nabla_{\mathcal{M}} p(x) \rangle_\bg+o(h),
\end{align*}
where the last estimation is uniform in $x \in \mathcal{M}$ and comes from the second estimation in \eqref{Equation: second inequality on K} in Lemma \ref{Lemma: an inequality on bounded variation functions}.
Thus, we have proved Equation \eqref{Equation: an integral that we have to change} for $J_1$.\\
The proof for $J_2$, given by  \eqref{Equation: an estimation has to be done}, is similar to what we have done for $J_1$. For identifying the Laplace-Beltrami operator in the last step of the proof, we use Equation \eqref{Equation: second equation in the auxiliary lemma} of Lemma \ref{Lemma: Some auxiliary calculations.} and the second point of Proposition \ref{Proposition: some calculations involving Laplacians and gradientsA}.
Therefore, we have proved Proposition \ref{Theorem: convergence of averaging kerel operators}.\hfill $\Box$

\subsection{Proof of Proposition \ref{Theorem: convergence of averaging kerel operators 2}}\label{sec:proof_Prop_2.3}

Let us now prove Proposition  \ref{Theorem: convergence of averaging kerel operators 2}. This proposition deals with the difference between the geodesic distance on $\cM$ and the Euclidean norm of $\R^m$.\\

By Inequalities \eqref{Equation: rho>c_1 inequality} and \eqref{Equation: rho>c_1 inequality_bis} of Lemma \ref{Lemma: An useful estimation}, we know that uniformly in $x$, when $h$ converges to $0$,
$$\int_{ \mathcal{M}} \PAR{K\left( \frac{\| x-y\|_2}{h}\right)+K\left( \frac{\rho( x-y)}{h}\right)} \mathbb{1}_{ \rho(x,y) \ge c_1} \mu( \text{d}y) =o(h^{d+3}).$$
Thus, by regularity of $f$, boundedness of $p$ and compactness of $\cM$, uniformly in $x$, when $h$ converges to $0$,
\begin{multline*}\int_{ \mathcal{M}} \PAR{K\left( \frac{\| x-y\|_2}{h}\right)+K\left( \frac{\rho( x-y)}{h}\right)} |f(x)-f(y)| {p(y)}\mathbb{1}_{ \rho(x,y) \ge c_1} \mu( \text{d}y)\\ =o(h^{d+3}).\end{multline*}
So, 
we only have to prove that uniformly in $x$,
\begin{multline*}\int_{ \mathcal{M}} \left| K\left( \frac{\rho(x,y)}{h}\right)-K\left( \frac{\| x-y\|_2}{h}\right) \right| |f(x) -f(y)| {p(y)} \mathbb{1}_{ \rho(x,y) < c_1} \mu( \text{d}y) \\=O(h^{d+3}),\end{multline*}
Or equivalently, using the change of variables $v=\mathcal{E}_x^{-1}(y)$ and $\rho(x,y)=\|\mathcal{E}_x^{-1}(y)\|_2$ by \eqref{eq:rho=normePhi},
\begin{align*}
	\int_{B_{\mathbb{R}^d(0,c_1)}} \left| K \left(\frac{\|v\|_2}{h} \right)- K \left(\frac{\|\mathcal{E}_x(v)-x\|_2}{h} \right)\right| &\left| f\circ \mathcal{E}_x(v)- f(x)\right|{p\circ \mathcal{E}_x(v)}
	\\
	&\times \sqrt{ \text{det} \widehat{g}^{x}_{ij}(v)} \text{d}v= O(h^{d+3}).
\end{align*} 
Besides, by regularity of $f$, boundedness of $p$ and compactness of $\cM$,  there is a constant $c$ such that $|f(x)-f(y)| \le c\|x-y\|_2 $. Moreover, by Property ii. of Theorem \ref{Theorem: Existence of a "good" family of normal coordinate systemsA}, the function $v\mapsto\text{det} \widehat{g}^{x}_{ij}(v)$ is bounded on $B_{\mathbb{R}^d(0,c_1)}$. 
Hence, it is sufficient to show that uniformly in $x$,
$$I:=\int_{B_{\mathbb{R}^d(0,c_1)}} \left| K \left(\frac{\|v\|_2}{h} \right)- K \left(\frac{\|\mathcal{E}_x(v)- x\|_2}{h} \right)\right| \| \mathcal{E}_x(v)-x\|_2 \text{d}v=O(h^{d+3}).$$
Recall that $\|\mathcal{E}_x(v)- x\|_2\leq\|v\|_2$ (by Theorem \ref{Theorem: Existence of a "good" family of normal coordinate systemsA}). 
By Inequation \eqref{Equation: first inequality on K} in Lemma \ref{Lemma: an inequality on bounded variation functions}, we have 
\begin{align*}
	I\leq &\int_{B_{\mathbb{R}^d(0,c_1)}} \left(  \int_{ \left( \frac{\|\mathcal{E}_x(v)- x\|_2}{h}, \frac{\|v\|_2}{h}\right]} \text{d}H(a)     \right) \| v\|_2 \text{d}v
	\\
		= &\int_{B_{\mathbb{R}^d(0,c_1)}}  \left( \int_{ \mathbb{R}_+}  \mathbb{1}_{ \|\mathcal{E}_x(v)- x\|_2 < ah \le \|v\|_2}      \text{d}H(a)  \right)\| v\|_2 \text{d}v.
\end{align*}

Also by Theorem \ref{Theorem: Approximation inequality for Riemannian distanceA}, there exists a constant $c_3$ such that $\forall x,y\in\cM$, $\rho(x,y)\leq c_3\NRM{x-y}_2^3+\NRM{x-y}$. The polynomial function $z\mapsto z+c_3z^3$ is an increasing bijective function  and we denote by $\varphi$ its inverse. Thus, for all $x,y \in \mathcal{M}$, $ \varphi( \rho(x,y))\le \|x-y\|_2 $.\\
Consequently, introducing $\varphi(\rho(x,\mathcal{E}_x(v)))=\varphi(\|v\|_2)$, we deduce 
\begin{align*}
I	\le & \int_{B_{\mathbb{R}^d(0,c_1)}}  \left( \int_{ \mathbb{R}_+}  \mathbb{1}_{  \varphi(\|v\|_2) < ah \le \|v\|_2}      \text{d}H(a)  \right)\| v\|_2 \text{d}v
	\\
	= & \int_{\mathbb{R}^+} \left( \int_{B_{\mathbb{R}^d(0,c_1)}} \| v\|_2.\mathbb{1}_{ ah  \le \|v \|_2  < ah+c_3(ah)^3} \text{d}v \right)\text{d}H(a),
\end{align*}
by Fubini's Theorem.
Finally,  using the spherical coordinate system as in the proof of Lemma \ref{Lemma: An useful estimation}, we see that:
\begin{align*}
I	\leq & 
	S_{d-1}\int_{\mathbb{R}^+} \left( \int_{0}^{c_1} r^{d}\mathbb{1}_{ ah  \le r  < ah+c_3(ah)^3} \text{d}r \right)\text{d}H(a)
	\\
	\le &	S_{d-1}\int_{\mathbb{R}^+} \left( \mathbb{1}_{ ah  \le c_1} \times \int_{ah}^{ah+c_3(ah)^3} r^{d} \mathrm{d}r \right)\text{d}H(a)
	\\
	\le & S_{d-1}\int_{\mathbb{R}^+} \left( \mathbb{1}_{ ah  \le c_1} \times c_3(ah)^3\left[ah+c_3(ah)^3\right]^d \right)\text{d}H(a)
	\\
	\le & S_{d-1} \int_{\mathbb{R}_+}  c_3(ah)^{d+3}(1+c_3c_1^2)^{d} \text{d}H(a)
	\\
	=& S_{d-1} c_3(1+c_3c_1^2)^{d} h^{d+3}\int_{\mathbb{R}^+}  a^{d+3} \text{d}H(a).
\end{align*}
This ends the proof of Proposition  \ref{Theorem: convergence of averaging kerel operators 2}.\hfill $\Box$

\section{Approximations by random operators}\label{Section: approximations by random operators}

In this section, we study the statistical error and
prove Proposition \ref{Prop:statistical_error}.

\begin{notation}
	For a $\mathcal{C}^3$-function $f: \cM  \rightarrow \mathbb{R}^{k}$, we denote respectively by $\|f'\|_{\infty}, \|f''\|_{\infty}, \|f'''\|_{\infty}$ the standard norm of multi-linear maps, i.e.
	\[
	\|f''\|_{\infty}=\sup_{\underset{\|v\|_2\le 1,\|w\|_2 \le 1}{x\in \cM, (v,w)\in\PAR{\dR^m}^2}}\ABS{f''(x)(v,w)}
	\]
	Recall that for $\alpha\in [\![1,m]\!] $ and $x\in\R^m$, we denote by $x^{\alpha}$ the $\alpha$-th coordinate of $x$. 
\end{notation}
Let us consider the following collection $\mathcal{F}$ of $\mathcal{C}^3$-functions.
\begin{align}\label{eq:def-cF}
\mathcal{F} & :=\{ f \in \mathcal{C}^{3}(\cM): \|f\|_{\infty} \le 1,\|f'\|_{\infty} \le 1, \|f''\|_{\infty }\le 1, \|f'''\|_{\infty} \le 1\}
\end{align}

Let $X$ be a random variable with distribution $p(x)\mu(\dd x)$ on $\cM$. We introduce the following sequence of random variables $(Z_n,n \in \mathbb{N})$:
\begin{align*}
Z_n &:=\sup_{f \in \mathcal{F} }\sup_{x \in  \mathcal{M}} \bigg|\mathcal{A}_{h_n,n}(f)(x) -\mathbb{E}[\mathcal{A}_{h_n,n}(f)(x)]  \bigg| 
		\\
&=\frac{1}{nh_n^{d+2}} \sup_{f \in \mathcal{F}} \sup_{x \in \mathcal{M}} \left| \sum_{i=1}^n \bigg( K\left( \frac{ \|X_i-x\|_2}{h_n}\right) (f(X_i)-f(x))   \right. \\ & \quad\quad\quad\quad\quad\quad\quad \left. -\mathbb{E}\SBRA{K\PAR{ \frac{ \|X-x\|_2}{h_n}} (f(X)-f(x))  }\bigg)\right| .
	\end{align*}
Recall that for all function $f$ and point $x$, $\mathbb{E}[\mathcal{A}_{h_n,n}(f)(x)]= \mathcal{A}_{h_n}(f)(x)$. 
We want
to prove that with probability $1$, 
\begin{equation}\label{Equation: convergence of E_n}
	Z_n= O \left( \sqrt{\frac{ \log h_n^{-1}}{nh_n^{d+2}}} + h_n\right).
\end{equation}
The general idea to prove this estimation is that instead of proving directly this convergence speed for $(Z_n)$, we show that its expectation also has this speed of convergence, that is:
\begin{equation} \label{Equation: convergence of expectations of E_n}
	\limsup_{n \rightarrow \infty} \left[ \left( \sqrt{\frac{ \log h_n^{-1}}{nh_n^{d+2}}} + h_n\right)^{-1} \mathbb{E}(Z_n) \right] <\infty,
\end{equation}
then \eqref{Equation: convergence of E_n} will follow easily from Talagrand's inequality (see Corollary \ref{Corollary: a corollary of Massart's} in Appendix) and Borel-Cantelli's theorem, as explained in Section \ref{Section: Step III: Conclusion}. The detailed plan for the proof of  \eqref{Equation: convergence of E_n} is as follows:
\begin{itemize}
	\item[Step I:] Use Taylor's expansion to divide $Z_n$ into many simpler terms each.
	\item[Step II:] Use Vapnik-Chernonenkis theory and Theorem \ref{gine's version} to bound the expectation of each terms.
	\item[Step III:] Use Talagrand's inequality to conclude.
\end{itemize}

{After using Talagrand's inequality, we have a non-asymptotic estimation of \[\mathbb{P}\left( \sup_{f \in \mathcal{F} }\sup_{x \in  \mathcal{M}} \bigg|\mathcal{A}_{h_n,n}(f)(x) -\mathbb{E}[\mathcal{A}_{h_n,n}(f)(x)]  \bigg|  \ge \delta\right) \]
for some suitable constant $\delta$ and will be able to prove the Corollary \ref{cor:dev_Ahn-A} at the end of this section. This term is of interest of many papers \cite{Hein2007,HeinAudibertvonLuxburg_COLT2005,Calder2022}.
}

\subsection{About the Vapnik-Chernonenkis theory}\label{sec:RappelVC}

Before starting the proof, we first recall here the main definitions and an important result of the Vapnik-Chernonenkis theory for the Borelian space $(\R^m,\mathcal{B}(\R^m))$ we will need. Other useful results are given  in Appendix \ref{Appendix:ConcentrationIneq}. For more details on the Vapnik-Chernonenkis theory, we refer the reader to \cite{devroyegyorfilugosi,Gine2016,nollan1987}. In this section, we will recall upper-bounds that exist for
\[	\sup_{f\in \mathcal{F}}\mathbb{E}\SBRA{ \left| \sum_{i=1}^n \PAR{ f(X_i)- \mathbb{E}\SBRA{f(X_i)}} \right|} \]
when the functions $f$ range over certain VC classes of functions that are defined below.\\

Let $(T,d)$ be a pseudo-metric space. Let $\varepsilon>0$ and $N\in\mathbb{N}\cup \{+\infty\}$. A set of points $\{ x_1,\dots, x_N\}$ in $T$ is an $\varepsilon$-cover of $T$ if for any $x \in T$, there exists $i \in [1,N]$ such that $ d( x, x_i) \le \varepsilon$. Then, the $\varepsilon$-\textit{covering number} of $T$ is defined as:
		\begin{align*}N(\varepsilon, T,d) := &\inf\{ N \in  \mathbb{N}\cup\{ +\infty\} : \text{there are }N \text{ points in } T \\ &\hskip 1cm \text{such that they form an }\varepsilon\text{-cover of } (T,d)    \}.
		\end{align*}

For a collection of real-valued measurable functions $\mathcal{F}$ on $\mathbb{R}^m$, a real measurable function $F$ defined on $\mathbb{R}^m$ is called \textit{envelope} of $\mathcal{F}$ if for any $x \in \mathbb{R}$, 
		\[
		 \sup_{f \in \mathcal{F}} |f(x)|\le F(x) .
		\]
This allows us to define VC classes of functions (see Definition 3.6.10 in \cite{Gine2016}). Recall that for a probability measure $Q$ on the measurable space $(\mathbb{R}^m,\mathcal{B}(\mathbb{R}^m))$, the $L^2(Q)$-distance given by \[\displaystyle{(f,g)\mapsto\PAR{ \int\ABS{f(x)-g(x)}^2Q(\dd x)}^{1/2}}\]defines a pseudo-metric on the collection of all bounded real measurable functions on $\mathbb{R}^m$.

\begin{definition}[VC class of functions, ] \label{def:VCclass}
		A class of measurable functions $\mathcal{F}$ is of VC type with respect to a measurable envelope $F$ of $\mathcal{F}$ if there exist finite constants $A, v$ such that for all probability measures $Q$ and $\varepsilon\in(0,1)$
		$$N( \varepsilon \|F\|_{L^2},\mathcal{F}, L^2(Q)) \le (A/\varepsilon)^v.$$
We will denote:
		\[N(\varepsilon, \mathcal{F}):=\sup_{Q}	N(\varepsilon, \mathcal{F}, L^2(Q)).\]
	\end{definition}

We now present a version of the  useful inequality (2.5) of Giné and Guillou in \cite{gine2001} that gives a bound for the expected concentration rate.
For a class of function $\cF$, let us define for any real valued function $\varphi: \mathcal{F} \rightarrow \mathbb{R}$,
		$$\| \varphi(f) \|_{ \mathcal{F}}=\sup_{f \in \mathcal{F}}| \varphi(f)|.$$

	\begin{theorem} (see \cite[Proposition 2.1 and Inequality (2.5)]{gine2001})
		\label{gine's version}
		Consider $n$ i.i.d random variables $X_1,\dots,X_n$ with values in $(\mathbb{R}^m, \cB(\mathbb{R}^m))$.\\
		Let $\mathcal{F}$ be a measurable uniformly bounded VC-type class of functions on $(\mathbb{R}^m, \cB(\mathbb{R}^m))$. We introduce two positive real number $\sigma^2$ and $U$, such that
		\[
		\sigma^2 \ge \sup_{f \in \mathcal{F}} \text{Var}\big(f(X_1)\big), \quad U \ge \sup_{f \in \mathcal{F}}\|f \|_{\infty}\quad \text{ and }\quad0 < \sigma \le 2U.
		\]
		Then there exists a constant $R$  depending only on the VC-parameters $A,\, v$ of $\cF$ and on $U$, such that:
	    \[
		\mathbb{E}\SBRA{ \left\| \sum_{i=1}^n \PAR{ f(X_i)- \mathbb{E}\SBRA{f(X_i)}} \right\|_{ \mathcal{F}}} \le R \PAR{ \sqrt{ n} \sigma \sqrt{ \ABS{\log\sigma}}+ \ABS{\log\sigma}}.
		\]
	\end{theorem}

{Notice that there exists also a formulation of the previous result in term of deviation probability (see e.g. \cite[Theorem 3]{PascalMassart2000}), that would lead to results similar to the ones established in \cite{Calder2022}.}
\subsection{Step I: decomposition of $Z_n$}
\label{Section: Step I: decomposition of $Z_n$}
We  first upper bound the quantity $Z_n$ with a sum of simpler terms.


\begin{lemma}For any function $f\in\cF$, there is a constant $c>0$ such that for all $n\geq 1$,
\label{Lemma: a majoring Lemma for Z_n}
\begin{equation}
	nh_n^{d+2}Z_n \le \sum_{\alpha =1}^{m} Y^{\alpha}_n + \sum_{\alpha ,\beta=1}^{m} Y^{\alpha,\beta}_n + Y^{(3)}_n +2 nch_n^{d+3},
\end{equation}
where
	\begin{align*}
	Y^{\alpha}_n &:= \sup_{x \in \mathcal{M}}\left| \sum_{i=1}^n \bigg[ K\left( \frac{ \|X_i-x\|_2}{h_n}\right) (X_{i}^{\alpha}-x_{i}^{\alpha}) -\right. 
	\\ &\hspace{6cm}
	\left.\mathbb{E}\left(K\left( \frac{ \|X-x\|_2}{h_n}\right) (X^{\alpha}-x^{\alpha})  \right)\bigg]\right|
	\\
	Y^{\alpha,\beta}_n &:=  \sup_{x \in \mathcal{M}}\left| \sum_{i=1}^n \bigg[K\left( \frac{ \|X_i-x\|_2}{h_n}\right) (X_{i}^{\alpha}-x^{\alpha})(X_{i}^{\beta}-x^{\beta}) -\right. \\ 
	&\hspace{4cm} \left.\mathbb{E}\left(K\left( \frac{ \|X-x\|_2}{h_n}\right) (X^{\alpha}-x^{\alpha})(X^{\beta}-x^{\beta})  \right) \bigg]\right| \quad \quad \quad 
	\\
	Y^{(3)}_n &:=  \sup_{x \in \mathcal{M}} \left| \sum_{i=1}^n K\left( \frac{ \|X_i-x\|_2}{h_n}\right)\|X_i-x\|_2^3 -\right. \\ 
	&\hspace{6cm} \left. \mathbb{E}\left[K\left( \frac{ \|X-x\|_2}{h_n}\right)\|X-x\|_2^3\right] \right|.
\end{align*}
\end{lemma}

\begin{proof}
Since for any $f\in\cF$, the differentials up to third order have operator norms bounded by $1$, then,  by the Taylor's expansion theorem, for any $(x,y) \in (\mathbb{R}^m)^2$, we have
\begin{align*}
f(y)-f(x)&=f'(x)(y-x)+ \frac{1}{2}f''(x)(y-x,y-x)+\tau_f(y;x)
\end{align*}
where $\tau_f$ is some function satisfying
\begin{equation} \sup_{f \in \mathcal{F}} |\tau_f(y,x)| \le \| f'''\|_{\infty} \|x-y\|_2^3= \|x-y\|_2^3.
\label{third derivative ineq}
\end{equation}

Thus, using the notation of the lemma, we deduce
\begin{equation*}
	nh_n^{d+2}Z_n   \le   \sum_{\alpha =1}^{m} Y^{\alpha}_n + \sum_{\alpha =1,\beta=1}^{m} Y^{\alpha,\beta}_n +Y^{r}_n,
\end{equation*}
with 
\begin{align*}
    Y^{r}_n &:= \sup_{\underset{x\in\cM}{f \in \mathcal{F}}}  \left| \sum_{i=1}^n \left(K\left( \frac{ \|X_i-x\|_2}{h_n}\right)\tau_f( X_i,x) - \mathbb{E}\left[K\left( \frac{ \|X-x\|_2}{h_n}\right)\tau_f(X,x) \right]\right) \right|.
\end{align*}

Using \eqref{third derivative ineq}, we now control $Y^r_n$ by $Y^{(3)}_n$,
as follows
	\begin{align*}
	Y^r_n &
	{\le} \sup_{x \in \mathcal{M}} \left| \sum_{i=1}^n K\left( \frac{ \|X_i-x\|_2}{h_n}\right)\|X_i-x\|_2^3 \right|
	\\
	&\hskip 2cm
	+ n\sup_{x \in \mathcal{M}} \mathbb{E}\left[K\left( \frac{ \|X-x\|_2}{h_n}\right)\|X-x\|_2^3\right]
	\\
	&\le Y^{(3)}_n+2n\sup_{x \in \mathcal{M}} \mathbb{E}\left[K\left( \frac{ \|X-x\|_2}{h_n}\right)\|X-x\|_2^3\right].
\end{align*}
Since the function $p$ is bounded on the compact $\cM$, using Inequation  \eqref{Equation: an useful estimation 3} of Lemma \ref{Lemma: An useful estimation}, we deduce that $Y^r_n\leq Y^{(3)}_n+2nch_n^{d+3}$, which conclude the proof.
\end{proof}

\subsection{Step II: Application of the Vapnik-Chernonenkis theory}

\subsubsection{Control the first order terms $\mathbb{E}[Y^{\alpha}_n ]$}
\label{Section: Control the first order terms}

Let $\alpha\in [\![1,m]\!]$ be fixed. Given the kernel $K$,  to bound the first order term $Y^{\alpha}_n$, we introduce three families of real functions on $\cM$ :
\begin{align*}
	&\mathcal{G}:=\left\{ \varphi_{h,y,z} : y,z \in \mathcal{M} , h>0 \right\},
	\quad
	\mathcal{G}_1:=\{ \psi_{h,y}  : y\in \mathcal{M},h >0 \} \\
	&\text{ and }
	\mathcal{G}_2:=\{ \zeta_y(x) : y \in \mathcal{M} \},
\end{align*}
with
\[
    \begin{array}{llcl}
	\varphi_{h,y,z} :&x&\longmapsto&K\left( \frac{\| x-y\|_2}{h} \right) (x^{\alpha}-z^{\alpha})\\
	\psi_{h,y} :&x&\longmapsto& K\left( \frac{\| x-y\|_2}{h}  \right)\\
	\zeta_y :&x&\longmapsto& x^{\alpha}-y^{\alpha}. 
	\end{array}
\]
Since $K$ is of bounded variation, by \cite[Lemma 22]{nollan1987}, $\mathcal{G}_1$ is VC-type w.r.t a constant envelope. Since $\cM$ is a compact manifold, by Lemma \ref{construction 2}, $\mathcal{G}_2$ is VC-type wrt to a constant envelope. Thus, using Lemma \ref{covering numbers of product spaces}, we deduce that $\mathcal{G}$ is a VC-type class of functions because  $\mathcal{G}=\mathcal{G}_1 \cdot \mathcal{G}_2$.  So, by Definition \ref{def:VCclass}, there exist real values $A \ge 6, v \ge 1$ depending only on the VC-characteristics of $\mathcal{G}_1$ and $\mathcal{G}_2$ such that, for all $ \varepsilon \in(0,1)$,
$$N( \varepsilon, \mathcal{G}) \le \left( \frac{A}{2\varepsilon}\right)^v.$$
Now, let us consider the following sequence of families of real functions on $\cM$: 
\[
\mathcal{H}_n=\left\{ \varphi_{n,y}  : y \in \cM \right\}, \text{ with }\varphi_{n,y} :
     x \longmapsto K\left(\frac{\| x-y\|_2}{h_n} \right) (x^{\alpha}-y^{\alpha}) .
\]

\begin{proposition}
\label{Proposition: control over the first order term}
Let $(X_i)_{i\geq 1}$ be a sample of i.i.d. random variables with distribution $p(x)\mu(\dd x)$ on the compact manifold $\cM$ and $X$ a random variable with the same distribution. We assume that $p$ is bounded on $\cM$. 

Then, if the kernel $K$ satisfies Assumption \ref{hyp:K} and the sequence $(h_n)_{n\geq 0}$ satisfies Assumption \eqref{hyp:h_th11}, we have
\begin{align}\label{ concentration in expection}
	\frac{1}{nh_n^{d+2}}			\mathbb{E}\SBRA{ \left\| \sum_{i=1}^n \PAR{f(X_i)-\mathbb{E}[f(X)]} \right\|_{\cH_n}}  = O\left(  \sqrt{ \frac{ \log h_n^{-1}}{nh_n^{d+2}}}\right).
\end{align}
\end{proposition}
\begin{proof}
Since $\mathcal{H}_n\subset \mathcal{G}$, by Lemma \ref{covering number of subspace }, for all $n$, we have $N(\varepsilon, \mathcal{H}_n)\le  \left(\frac{A}{\varepsilon}\right)^v$.
Hence, by theorem \ref{gine's version}, there exists a constant $R$  depending only on $A,\,v$ and $U$ such that:
\begin{align*}
	\mathbb{E}\SBRA{ \left\| \sum_{i=1}^n f(X_i)-\mathbb{E}(f(X)) \right\|_{\cH_n}}   \le R\left( \sqrt{n}\sigma\sqrt{\ABS{\log\sigma}} \text{ } {+} \text{ }  \ABS{\log\sigma}\right).
\end{align*}
where $U$ is a constant such that $ U \ge \sup_{f \in \mathcal{H}_n} \|f\|_{\infty}$, and  $\sigma$ is a constant such that $4U^2 \ge \sigma^2 \ge \sup_{f \in \mathcal{H}_n} \mathbb{E}[f^2(X)]$.\\
Since $\mathcal{H}_n\subset \mathcal{G}$,  we can choose $U$ to be the constant envelope of $\mathcal{G}$ (thus, independent of $n$). Besides, we see that:
\begin{align*}
\sup_{f \in \mathcal{H}_n}\mathbb{E}\left[ f^2(X)\right]& \le \|K\|_{\infty} \sup_{x \in \mathcal{M}} \int_{\mathcal{M}} K\left(\frac{\| x-y\|}{h_n} \right)(x^{\alpha}-y^{\alpha})^2 p(y)\mu(dy).
\end{align*}
By Inequation \ref{Equation: an useful estimation 2} of Lemma \ref{Lemma: An useful estimation}, we deduce that, there is $c>0$ such that 
\begin{equation}\label{eq:bound-Hn}
\sup_{f \in \mathcal{H}_n}\mathbb{E}\left[ f^2(X)\right]
	 {\le}  \|K\|_{\infty}\|p\|_{\infty} \mu(\mathcal{M})c h_n^{d+2},
	 \end{equation}
	 which goes to $0$ when $n\to +\infty$.
Choose $\sigma^2 := \sigma_n^2=\|K\|_{\infty}\|p\|_{\infty} \mu(\mathcal{M})c h_n^{d+2} $. For $n$ large enough, $\sigma_n \le 2U$. Hence, using Assumption \eqref{hyp:h_th11} on the sequence $(h_n)_{n\geq 1}$, we deduce
\begin{align*}
	\frac{1}{nh_n^{d+2}}			\mathbb{E} \SBRA{\NRM{ \sum_{i=1}^n f(X_i)-\mathbb{E}[f(X)] }_{\cH_n}} &=O\left(  \sqrt{ \frac{ \log h_n^{-1}}{nh_n^{d+2}}} {+}  \frac{ \log h_n^{-1}}{nh_n^{d+2}} \right) \\ &= O\left(  \sqrt{ \frac{ \log h_n^{-1}}{nh_n^{d+2}}}\right).
\end{align*}
This concludes the proof. \end{proof}

The conclusion of the above proposition means that:
\[
\mathbb{E}[Y^{\alpha}_n ]= O\left(  \sqrt{ \frac{ \log h_n^{-1}}{nh_n^{d+2}}}\right).
\]

\subsubsection{Control the second order terms $\mathbb{E}\SBRA{Y^{\alpha,\beta}_n}$}
\label{Section: Control the second order terms}
The way to bound the second order term $Y^{\alpha,\beta}_n$, for $\alpha,\beta \in [\![ 1,m ]\!]$, is similar to the previous step, but instead of considering $\mathcal{H}_n$, we consider the following VC-type family of functions:
$$\mathcal{I}_n:=\left\{ \xi_{n,y,z}:x \mapsto K\left(\frac{\| x-y\|}{h_n} \right) (x^{\alpha}-y^{\alpha})(x^{\beta}-q^{\beta})  : y \in \mathcal{M}, z\in \mathcal{M} , \right\}.$$
We notice that, for any r.v. $X$,
\[
\mathbb{E}\SBRA{\sup_{g\in \mathcal{I}_n}\ABS{g^2(X)}}
\leq \text{diam}(\cM)^2\mathbb{E}\SBRA{\sup_{f\in \mathcal{H}_n}\ABS{f^2(X)}}.
\]
Using \eqref{eq:bound-Hn}, we deduce $\sup_{g \in \mathcal{I}_n} \mathbb{E}[g^2(X)] = O( h_n^{d+2}  ),$ and  
\begin{align*} \frac{1}{nh_{n}^{d+2}} \mathbb{E} \sup_{g \in \mathcal{I}_n} \left| \sum_{i=1}^{n} \left( g(X_i)-\mathbb{E}[g(X_i)] \right)\right| = O\left(  \sqrt{ \frac{ \log h_n^{-1}}{nh_n^{d+2}}}\right).
\end{align*}

Therefore, we conclude that:
\[
\mathbb{E}[Y^{\alpha,\beta}_n ]= O\left(  \sqrt{ \frac{ \log h_n^{-1}}{nh_n^{d+2}}}\right).
\]

\subsubsection{Control the third order terms $\mathbb{E}\SBRA{Y^{(3)}_n}$}
	This step is essentially the same as the two previous steps, except that the considered family of functions is a little bit different, which is:
$$\mathcal{K}_n:= \left\{ x \mapsto K\left(\frac{\| x-y\|}{h_n} \right) \|x-y\|^3 : y \in \mathcal{M} \right\}.$$
With the same arguments as before, we obtain:
\[
\mathbb{E}\SBRA{Y^{(3)}_n }= O\left(  \sqrt{ \frac{ \log h_n^{-1}}{nh_n^{d+2}}}\right).
\]

Now, thanks to Step I, Step II and Lemma \ref{Lemma: a majoring Lemma for Z_n}, we have shown that:
\begin{equation}\label{eq:ordreZn}\mathbb{E}[Z_n]= O\left(   \sqrt{ \frac{ \log h_n^{-1}}{nh_n^{d+2}}}  \text{  } {+} h_n \right).\end{equation}


\subsection{Step III: Conclusion}
\label{Section: Step III: Conclusion}
Recall that the set of function $\cF$ is defined by \eqref{eq:def-cF}. 
Since $p$ is bounded on $\cM$, by \eqref{Equation: an useful estimation 4} of Lemma \ref{Lemma: An useful estimation}, there exists $c>0$ such that $\forall f \in \mathcal{F}, \forall x \in \mathcal{M}$
\begin{align*}&\mathbb{E}\left[  K\left( \frac{ \|X-x\|}{h_n}\right)^2 ( f(X)-f(x))^2 \right] \\
	&\le \|K\|_{\infty}   \,  \mathbb{E}\left[  K\left( \frac{ \|X-x\|}{h_n}\right) \| X -x\|^2  \right]\stackrel{}{\le} {\|K\|_{\infty}}  c h_n^{d+2}.
\end{align*}
In other words,
$$\sup_{f \in \mathcal{F}} \sup_{x \in \mathcal{M}} \mathbb{E}\left[ K\left( \frac{ \|X-x\|}{h_n}\right)^2 ( f(X)-f(x))^2\right]\le \|K\|_{\infty} ch_n^{d+2}.$$
Thus by choosing $\sigma:= \sigma_n =  \sqrt{ \|K\|_{\infty}ch_n^{d+2}}$,  and using Massart's version of Talagrands' inequality (c.f. Corollary \ref{Corollary: a corollary of Massart's}) with the functions of the form $y\mapsto K\PAR{\NRM{y-x}_2\over h_n}\PAR{f(y)-f(x)}$, for all $n$ sufficiently large and any positive number $t_n>0$, with probability at least $1-e^{-t_n}$,
\begin{multline}\label{etape:dev_finie} \sup_{f \in \mathcal{F} }\sup_{y \in \mathcal{M}} nh_n^{d+2}\ABS{\mathcal{A}_{h_n,n}(f)(x) -\mathbb{E}\SBRA{\mathcal{A}_{h_n,n}(f)(x)}} \\ \le 9\left( nh_n^{d+2}\mathbb{E}[Z_n] +\sigma_n\sqrt{nt_n }  +bt_n \right).\end{multline}
where in this case, the constant envelope $b$ is equal to  $$b:=  \|K\|_\infty \text{diam}{\mathcal{M}}.$$

Choose $t_n=2 \log n$, by Borel-Catelli's lemma, with probability 1
$$\sup_{f \in \mathcal{F} }\sup_{x \in \mathcal{M}} \bigg|\mathcal{A}_{h_n,n}(f)(x) -\mathbb{E}[\mathcal{A}_{h_n,n}(f)(x)]\bigg| =O\left(  \sqrt{ \frac{ \log h_n^{-1}}{nh_n^{d+2}}}   {+} h_n {+} \sqrt{ \frac{\log n}{n h_n^{d+2}}} \right).$$

Besides, under Assumption \eqref{hyp:h_th11} on the sequence $(h_n)_{n\geq 1}$, $\displaystyle{\lim_{n \rightarrow +\infty} {nh_{n}^{d+2}}= +\infty}$, hence $\log h_n^{-1} = O(\log n)$.
Thus with probability 1,
$$\sup_{f \in \mathcal{F} }\sup_{x \in \mathcal{M}} \bigg|\mathcal{A}_{h_n,n}(f)(x) -\mathbb{E}[\mathcal{A}_{h_n,n}(f)(x)]\bigg|  =O\left(  \sqrt{ \frac{ \log h_n^{-1}}{nh_n^{d+2}}}  {+} h_n \right).$$
This ends the proof of Proposition \ref{Prop:statistical_error}. Hence, Theorem \ref{Theorem: main theorem} is proved.

\subsection{Proof of Corollary \ref{cor:dev_Ahn-A}}

Using the results of the above sections, we can now prove Corollary \ref{cor:dev_Ahn-A}. First, we see that by the proofs of Propositions \ref{Theorem: convergence of averaging kerel operators}, \ref{Theorem: convergence of averaging kerel operators 2} and \eqref{eq:ordreZn}, there is a constant $C>0$ such that for all $h>0, n \in \mathbb{N}:$
\begin{equation}\sup_{f \in \mathcal{F}}\sup_{x\in \cM}  \ABS{\mathbb{E}\SBRA{\mathcal{A}_{h,n}(f)(x)}-
    \mathcal{A}(f)(x)} \le Ch, \label{eq:EstimeeAf}\end{equation}
and
\begin{equation}\mathbb{E}[Z_n] \le C\left(\sqrt{ \frac{\log h_n^{-1}}{nh_n^{d+2}}}+h_n \right). \label{eq:estimeeEZ_n}
\end{equation}

Then by choosing $t_n:= \delta^2 nh_n^{d+2}$ in \eqref{etape:dev_finie} with $\delta \in [h_n\vee \sqrt{\frac{\log h_n^{-1}}{nh_n^{d+2}}},1]$, we know that with probability at least $1 -e^{-\delta^2 nh_n^{d+2}}$, 
\begin{equation*} 
 \sup_{f \in \mathcal{F} }\sup_{y \in \mathcal{M}} \ABS{\mathcal{A}_{h_n,n}(f)(x) -\mathbb{E}\SBRA{\mathcal{A}_{h_n,n}(f)(x)}}  \le \frac{9\left( nh_n^{d+2}\mathbb{E}[Z_n] +\sigma_n\sqrt{nt_n }  +bt_n \right) }{nh_n^{d+2}}.
\end{equation*}
Besides, by \eqref{eq:estimeeEZ_n}, we have:
\begin{multline*}
\frac{ nh_n^{d+2}\mathbb{E}[Z_n] +\sigma_n\sqrt{nt_n }  +bt_n }{nh_n^{d+2}} \\
\begin{aligned}
\le & C\left(\sqrt{ \frac{\log h_n^{-1}}{nh_n^{d+2}}}+h_n \right) + \frac{ \sigma_n\sqrt{nt_n }  +bt_n }{nh_n^{d+2}}
\\
= & C\left(\sqrt{ \frac{\log h_n^{-1}}{nh_n^{d+2}}}+h_n \right)+\sqrt{\|K \|_{\infty}c}\delta+ \|K\|_{\infty} (\text{dim}\cM) \delta^2
\\
 \le & \left( 2C+ \sqrt{\|K \|_{\infty}c}+\|K\|_{\infty} (\text{dim}\cM) \right) \delta .
\end{aligned}
\end{multline*} 
In addition, after \eqref{eq:EstimeeAf}, we have:
\begin{multline*}
\sup_{f \in \mathcal{F} }\sup_{x \in \mathcal{M}} \ABS{\mathcal{A}_{h_n,n}(f)(x) -\mathcal{A}(f)(x)}
\\
\begin{aligned}
 \le &  \sup_{f \in \mathcal{F} }\sup_{x \in \mathcal{M}} \ABS{\mathcal{A}_{h_n,n}(f)(x) -\mathbb{E}[\mathcal{A}_{h_n,n}(f)(x)]} \\
 & +\sup_{f \in \mathcal{F} }\sup_{x \in \mathcal{M}} \ABS{ \mathbb{E}[\mathcal{A}_{h_n,n}(f)(x)] - 
    \mathcal{A}(f)(x)} 
\\
 \le & \sup_{f \in \mathcal{F} }\sup_{x \in \mathcal{M}} \ABS{\mathcal{A}_{h_n,n}(f)(x) -\mathbb{E}[\mathcal{A}_{h_n,n}(f)(x)]}+Ch_n
\\
 \le & \sup_{f \in \mathcal{F} }\sup_{x \in \mathcal{M}} \ABS{\mathcal{A}_{h_n,n}(f)(x) -\mathbb{E}[\mathcal{A}_{h_n,n}(f)(x)]}+C \delta .
\end{aligned}
\end{multline*}
Therefore, by letting 
\begin{equation}\label{def:Cprime}
C':= 9[2C+\sqrt{\|K \|_{\infty}c}+\|K\|_{\infty} (\text{dim}\cM)]+C,
\end{equation}where $C$ is the constant appearing in \eqref{eq:EstimeeAf} and \eqref{eq:estimeeEZ_n}, we have:
\begin{multline*}
\mathbf{P}\left( \sup_{f \in \mathcal{F} }\sup_{x \in \mathcal{M}} \ABS{\mathcal{A}_{h_n,n}(f)(x) -\mathcal{A}(f)(x)} > C' \delta \right) 
\\
\begin{aligned}
\le & \mathbf{P} \left( \sup_{f \in \mathcal{F} }\sup_{x \in \mathcal{M}} \ABS{\mathcal{A}_{h_n,n}(f)(x) -\mathbb{E}[\mathcal{A}_{h_n,n}(f)(x)]} > C'\delta- C\delta\right) 
\\
 = &\mathbf{P} \Bigg( \sup_{f \in \mathcal{F} }\sup_{x \in \mathcal{M}} \ABS{\mathcal{A}_{h_n,n}(f)(x) -\mathbb{E}[\mathcal{A}_{h_n,n}(f)(x)] }\\
 & \hspace{2cm}  >9\left( 2C+ \sqrt{\|K \|_{\infty}c}+\|K\|_{\infty} (\text{dim}\cM)\right) \delta \Bigg)
\\
\le & \exp( -\delta^2 n h_n^{d+2}).
\end{aligned}
\end{multline*}
This proves Corollary \ref{cor:dev_Ahn-A}.

\section{Convergence of $k$NN Laplacians}\label{sec:knn}

We now consider the case of random walks exploring the $k$NN graph on $\cM$ built on the vertices $\{X_i\}_{i\geq 1}$, as defined in the introduction. 

Recall that for $n\in \N$, $k\in \{1,\dots n\}$ and $x\in \cM$, the distance between $x$ and its $k$-nearest neighbor is defined in \eqref{def:dist-kNN} and that the Laplacian of the $k$NN-graph is given by, for $x\in\cM$,
\begin{equation}
\mathcal{A}_{n}^{\kNN}(f)(x):=\frac{1}{nR_{n,k_n}^{d+2}(x)} \sum_{i=1}^n \ind_{[0,1]}\left(\frac{\|X_i-x\|_2}{R_{n,k_n}(x)}\right) (f(X_i)-f(x)).
\end{equation}
Notice here that the width of the moving window, $R_{n,k_n}(x)$, is random and depends on $x\in \cM$, contrary to  $h_n$ in the previous generator $\cA_{h_n,n}$ defined by \eqref{eq:def_Ahn-n}.\\

To overcome this difficulty, we use the result of Cheng and Wu \cite[Th. 2.3]{chengwu}, {with $h=\ind_{[0,1]}$,} that allows us to control the randomness and locality of the window:
\begin{theorem}[Cheng-Wu, Th. 2.3]\label{th:ChengWu}
Under Assumption \ref{hyp:K}, if the density $p$ satisfies \eqref{hyp:p} and if 
\[\lim_{n\rightarrow +\infty}\frac{k_n}{n}=0,\quad \mbox{ and } \lim_{n\rightarrow +\infty}\frac{k_n}{\log(n)}=+\infty,\]
then, with probability higher than $1 - n^{-10}$,
\begin{equation}
    \sup_{x\in \cM} \ABS{\frac{R_{n,k_n}(x)}{V_d^{1/d} p^{-1/d}(x)\big(\frac{k_n}{n}\big)^{1/d}} - 1} = O
\PAR{\PAR{ \frac{k_n}{n}}^{2/d}+ \frac{3\sqrt{13}}{d}\sqrt{ \frac{\log n}{k_n}}},
\end{equation}
{where $V_d$ is the volume of unit $d-$ball.}
\end{theorem}

As a corollary for Theorem \ref{th:ChengWu}, we deduce that the distance $R_{n,k_n}(x)$ is, uniformly in $x$ and with large probability, of the order of $h_n$:
\begin{equation}\label{eq:borneR_nkn}
\P \PAR{ \forall x\in \cM,\ R_{n,k_n}(x)\in [ h_n(x)- \gamma_n , h_n(x)+\gamma_n]}\geq 1-n^{-10},
\end{equation}with
\begin{equation}\label{eq:choix_h_gamma}h_n(x)=V_d^{1/d} p^{-1/d}(x)\PAR{\frac{k_n}{n}}^{1/d},\,\mbox{ and }\, \gamma_n= 2  \PAR{\PAR{\frac{k_n}{n}}^{2/d}+ \frac{3\sqrt{13}}{d}\sqrt{ \frac{\log n}{k_n}}}.\end{equation}

We will then derive the limit Theorem \ref{main:theorem-kNN} for the rescaling of the $k$NN Laplacian using next result, proved right after. 

\begin{theorem}\label{Theorem: main theorem 2}
Suppose that the density of points $p$ on the compact smooth manifold $\mathcal{M}$ is of class $\cC^2$. 
Suppose that Assumptions \ref{hyp:K} for the kernel $K$ are satisfied and that $(h_n, n\in\mathbb{N})$ satisfies \eqref{hyp:h_th11}, i.e.
\begin{equation*}
    \lim_{n\rightarrow +\infty} h_n =0,\qquad \mbox{ and }\qquad \lim_{n\rightarrow +\infty}\frac{\log h_n^{-1}}{nh_n^{d+2}} =0.
\end{equation*}
Then, for all real number $\kappa>1$, with probability $1$,  for all  $f \in \mathcal{C}^{3}(\mathcal{M})$,
\begin{equation}
	\sup_{ \kappa^{-1}h_n \le r \le \kappa h_n}\sup_{x \in \mathcal{M}} \left|  \mathcal{A}_{r,n}(f)(x)- \mathcal{A}(f)(x)\right|= O\left(  \sqrt{ \frac{ \log h_n^{-1}}{nh_n^{d+2}}}  \toan{+} h_n \right),\label{eq:cv_th2}
\end{equation} 
where $\cA_{r,n}$ and $\cA$ are respectively defined by \eqref{eq:def_Ahn-n} (replacing $h_n$ with $r$) and \eqref{Equation: differential operator}.
\end{theorem}

\begin{proof}[Proof of Theorem \ref{main:theorem-kNN}]
Assume that Theorem \ref{Theorem: main theorem 2} is proved. {We know that the event $\{\forall x\in \cM,\ R_{n,k_n}(x)\in [ h_n(x)- \gamma_n , h_n(x)+\gamma_n]\}$ is of probability $1-n^{-10}$. Therefore, by Borel-Cantelli's theorem, with probability $1$, there exists $N:=N(\omega) \in \mathbb{N}$ such that: }
$$\forall n \ge N: \forall x\in \cM,\ R_{n,k_n}(x)\in [ h_n(x)- \gamma_n , h_n(x)+\gamma_n] $$
{Thus with probability $1$, for all $n \ge N(\omega)$, we have:}
\begin{align*}
\left|  \mathcal{A}_{n}^{\kNN}(f)(x)- \mathcal{A}(f)(x)\right| \leq & 
\sup_{r\in [a_n,b_n]}\left|  \mathcal{A}_{r,n}(f)(x)- \mathcal{A}(f)(x)\right|
\end{align*}with:
\begin{align*}
    a_n= & V_d^{1/d}p_{\max}^{-1/d}\PAR{\frac{k_n}{n}}^{1/d}-\gamma_n\\
    b_n= & V_d^{1/d}p_{\min}^{-1/d}\PAR{\frac{k_n}{n}}^{1/d}+\gamma_n.
\end{align*}Notice that for $n$ large enough, $a_n$ will be positive.
Using Theorem \ref{Theorem: main theorem 2} with $h_n=b_n$ and $\kappa =(p_{\max}/p_{\min})^{1/d}+1$, we see that $[a_n,b_n]\subset [\kappa^{-1}h_n,\kappa h_n]$. The result follows with the choice of number of neighbors $k_n$ in \eqref{hyp:k_n} coming from \eqref{hyp:h_th11} with our choice of $h_n$. The rate of convergence in \eqref{eq:rate_kNN} result from \eqref{eq:cv_th}. 
\end{proof}

\begin{proof}[Proof for Theorem \ref{Theorem: main theorem 2}]
The proof for the above theorem is essentially the same as the proof we presented for Theorem \ref{Theorem: main theorem} except some necessary modifications. Decomposing the error term as in \eqref{etape1}, we have to treat with similar terms. The approximations involving the geometry and corresponding to Propositions \ref{Theorem: convergence of averaging kerel operators} and \ref{Theorem: convergence of averaging kerel operators 2} can be generalized directly to account for a supremum in the window width $r\in [\kappa^{-1}h_n,\kappa h_n]$. Let us consider the statistical term.\\

We recall that $\cF$ is defined by \eqref{eq:def-cF}.We introduce the following sequence of random variables $(\tilde{Z}_n,n \in \mathbb{N})$:
\begin{align*}
\tilde{Z}_n &:=\sup_{f \in \mathcal{F} }\sup_{ \kappa^{-1}h_n \le r \le \kappa h_n}\sup_{x \in  \mathcal{M}} \bigg|\mathcal{A}_{r,n}(f)(x) -\mathbb{E}[\mathcal{A}_{r,n}(f)(x)]  \bigg| 
		\\
&=\frac{1}{nh_n^{d+2}} \sup_{f \in \mathcal{F}} \sup_{ \kappa^{-1}h_n \le r \le \kappa h_n}\sup_{x \in \mathcal{M}} \left| \sum_{i=1}^n \bigg( K\left( \frac{ \|X_i-x\|_2}{r}\right) (f(X_i)-f(x))   \right. \\ & \quad\quad\quad\quad\quad\quad\quad \left. -\mathbb{E}\SBRA{K\PAR{ \frac{ \|X-x\|_2}{r}} (f(X)-f(x))  }\bigg)\right| .
	\end{align*}
 Similar to what we did in Section \ref{Section: Step I: decomposition of $Z_n$}, we can show that there is a constant $c$ independent of $n$ such that:
 \begin{equation}
	nh_n^{d+2}\tilde{Z}_n \le \sum_{\alpha =1}^{m} \tilde{Y}^{\alpha}_n + \sum_{\alpha ,\beta=1}^{m} \tilde{Y}^{\alpha,\beta}_n + \tilde{Y}^{(3)}_n +2 nch_n^{d+3},
\end{equation}
where
	\begin{align*}
	\tilde{Y}^{\alpha}_n &:= \sup_{ \kappa^{-1}h_n \le r \le \kappa h_n}\sup_{x \in \mathcal{M}}\left| \sum_{i=1}^n \bigg[ K\left( \frac{ \|X_i-x\|_2}{r}\right) (X_{i}^{\alpha}-x_{i}^{\alpha}) -\right. 
	\\ &\hspace{6cm}
	\left.\mathbb{E}\left(K\left( \frac{ \|X-x\|_2}{r}\right) (X^{\alpha}-x^{\alpha})  \right)\bigg]\right|
	\\
	\tilde{Y}^{\alpha,\beta}_n &:=  \sup_{ \kappa^{-1}h_n \le r \le \kappa h_n}\sup_{x \in \mathcal{M}}\left| \sum_{i=1}^n \bigg[K\left( \frac{ \|X_i-x\|_2}{r}\right) (X_{i}^{\alpha}-x^{\alpha})(X_{i}^{\beta}-x^{\beta}) -\right. \\ 
	&\hspace{4cm} \left.\mathbb{E}\left(K\left( \frac{ \|X-x\|_2}{r}\right) (X^{\alpha}-x^{\alpha})(X^{\beta}-x^{\beta})  \right) \bigg]\right| \quad \quad \quad 
	\\
	\tilde{Y}^{(3)}_n &:=  \sup_{ \kappa^{-1}h_n \le r \le \kappa h_n}\sup_{x \in \mathcal{M}} \left| \sum_{i=1}^n K\left( \frac{ \|X_i-x\|_2}{r}\right)\|X_i-x\|_2^3 -\right. 
	\\ &\hspace{6cm}
	\left. \mathbb{E}\left[K\left( \frac{ \|X-x\|_2}{r}\right)\|X-x\|_2^3\right] \right|.
\end{align*}
We now treat these terms by applying Vapnik-Chernonenkis theory. Let us start with the control the first order terms $\mathbb{E}[\tilde{Y}^{\alpha}_n ]$:

 In Section \ref{Section: Control the first order terms}, we have already shown that the family
 \[
\mathcal{G}:=\left\{\varphi_{h,y,z}:x\longmapsto K\left( \frac{\| x-y\|_2}{h} \right) (x^{\alpha}-z^{\alpha})  : y,z \in \mathcal{M} , h>0\right\} 
\] is a VC class of functions, and that there exist real values $A \ge 6, v \ge 1$ such that, for all $ \varepsilon \in(0,1)$,
$N( \varepsilon, \mathcal{G}) \le \big( A/2\varepsilon\big)^v.$

Now, on top of this, we consider the following sequence of families of real functions on $\cM$: 
 \[
\tilde{\mathcal{H}}_n=\left\{ \varphi_{r,y}  : y \in \cM, \kappa^{-1} h_n \le r \le \kappa h_n \right\}, 
\]with $\varphi_{r,y} :
     x \longmapsto K\left(\frac{\| x-y\|_2}{r} \right) (x^{\alpha}-y^{\alpha})$.
Because each $\tilde{\mathcal{H}}_n$ is a subfamily of $\mathcal{G}$, it is still a VC class for which we can use the Talagrand inequality \ref{gine's version}. The latter can deal with the additional supremum with respect to the window width. Similarly to what we did in the proof of Proposition \ref{Proposition: control over the first order term}, we obtain that:
\begin{align*}
	\frac{1}{nh_n^{d+2}}			\mathbb{E}\SBRA{ \left\| \sum_{i=1}^n \PAR{f(X_i)-\mathbb{E}[f(X)]} \right\|_{\tilde{\cH}_n}}  = O\left(  \sqrt{ \frac{ \log h_n^{-1}}{nh_n^{d+2}}}\right),
\end{align*}
which means that as $n \rightarrow \infty$,
\[
\mathbb{E}[\tilde{Y}^{\alpha}_n ]= O\left(  \sqrt{ \frac{ \log h_n^{-1}}{nh_n^{d+2}}}\right).
\]

The control the second and third order terms are done as in Sections \ref{Section: Control the second order terms} and \ref{Section: Control the first order terms}, using the same trick and the classes of functions
\begin{multline*}\tilde{\mathcal{I}}_n:=\left\{ x \mapsto K\left(\frac{\| x-y\|}{r} \right) (x^{\alpha}-y^{\alpha})(x^{\beta}-q^{\beta})  :\right. \\
\left.y \in \mathcal{M}, q\in \mathcal{M} , \kappa^{-1}h_n \le r \le \kappa h_n\right\}
\end{multline*}
and
$$\tilde{\mathcal{K}}_n:= \left\{ x \mapsto K\left(\frac{\| x-y\|}{r} \right) \|x-y\|^3 : y \in \mathcal{M}, \kappa^{-1}h_n \le r \le \kappa h_n  \right\}.$$
This provides:
\[
\mathbb{E}[\tilde{Y}^{\alpha,\beta}_n ]= O\left(  \sqrt{ \frac{ \log h_n^{-1}}{nh_n^{d+2}}}\right),\quad \mbox{ and }\quad \mathbb{E}[\tilde{Y}^{(3)}_n ]= O\left(  \sqrt{ \frac{ \log h_n^{-1}}{nh_n^{d+2}}}\right).
\]
Therefore, we can deduce the conclusion  by using the same argument presented in Section \ref{Section: Step III: Conclusion}.
\end{proof}

{\footnotesize	
\bibliographystyle{acm}
\bibliography{Graph_laplacian.bib}
}

\newpage
\appendix

\section{Some concentration inequalities}\label{Appendix:ConcentrationIneq}
	\subsection{Talagrand's concentration inequality}
	As a corollary of Talagrand's inequality presented in Massart \cite[Theorem 3]{PascalMassart2000}, where for simplicity we choose $\varepsilon=8$, we have the following deviation inequality:
	\begin{corollary}[Simplified version of Massart's inequality]\label{Corollary: a corollary of Massart's}
	Consider $n$ independant random variables $\xi_1,\dots, \xi_n$ with values in some measurable space $ ( \mathbb{X}, \mathfrak{X})$. Let $\mathcal{F}$ be some countable family of real-valued measurable functions on $(\mathbb{X},\mathfrak{X})$ such that for some positive real number $b$, $\| f\|_{\infty} \le b$ for every $f \in \mathcal{F}$.
		$$Z:= \sup_{f \in\mathcal{F}} \left| \sum_{i=1}^n \bigg(f( \xi_i) - \mathbb{E} \left[ f(\xi_i)\right]\bigg)\right| .$$
		then with $\sigma^2 = \sup_{f \in \mathcal{F}}  \text{Var}( f(\xi_1))$, and for any positive real number $x$,
		$$\mathbb{P} \PAR{Z \ge 9(\mathbb{E}[Z]+\sigma\sqrt{nx} +bx)} \le e^{-x}  .$$
	\end{corollary}
	
	\subsection{Covering numbers and complexity of a class of functions}
	If $S \subset T$ is a subspace of $T$, it is not true in general that $N( \varepsilon, S,d) \le N( \varepsilon, T,d)$ because of the constraints that the cencers $x_i$ should belong to $S$. However, we can bound the covering number of $S$ by $T$'s as follows
	\begin{lemma}
		\label{covering number of subspace }
		If $S \subset T$ is a subspace of the metric space $(T,d)$, then for any positive number $\varepsilon$
		$$N( 2\varepsilon, S,d) \le N( \varepsilon, T,d). $$
	\end{lemma}
	\begin{proof}
		Let $\{x_1,...,x_N\}$ be a $\varepsilon$-cover of $T$ and for any $i \in [\![ 1,N ]\!]$, let us define $K_i:= \{ x \in T: d(x,x_i) \le \varepsilon\}$. 
		Of course, $K_i $ may not intersect $S$, hence, without loss of generality, assume that for a natural number $0<m \le N$ we have that
		$K_i \cap S \ne \emptyset$ if and only if $i \le m$. 
		Let $y_i$ be any point in $K_i \cap S$ for $i \in [\![1,m]\!]$. Since $\{x_1,...,x_N\}$ is a $\varepsilon$ cover of $T$, for any $y \in S$, there exists a $i \le m$ such that $y \in K_i \cap S$. Hence,
		$d(y,y_i) \le 2\varepsilon$.
		Consequently, $y_1,...,y_m$ be a $2\varepsilon$-cover of $(S,d)$.
	\end{proof}
	Let us consider the Borel space $( \R^m, \mathcal{B}(R^m))$. If $\mathcal{F}, \mathcal{G}$ are two collections of measurable functions on $\mathbb{X}$, we are interested in the "complexity" of $\mathcal{F} \cdot \mathcal{G} =\{ fg | f \in \mathcal{F} , g \in \mathcal{G}\}$.
	\begin{lemma}[Bound on $\varepsilon$-covering numbers]
		\label{covering numbers of product spaces}
		Let $\mathcal{F}, \mathcal{G}$ be two bounded collections of measurable functions, i.e, there are two constants $c_1,c_2$ such that 
		$$\|f\|_{\infty} \le c_1  \text{ and } \|g\|_{\infty} \le c_2 \text{ for all }f \in \mathcal{F} \text{ , } g \in \mathcal{G}.$$
		then for any probability measure $Q$, 
		$$N(  2\varepsilon c_1c_2,  \mathcal{F} \cdot \mathcal{G}, L^2(Q) ) \le N(  \varepsilon c_1,  \mathcal{F} ,L^2(Q)) N( \varepsilon c_2, \mathcal{G}, L^2(Q) ). $$
	\end{lemma}
	\begin{proof}
		If $f_1,f_2,...,f_n$ is a $\varepsilon c_1$-cover of $(\mathcal{F},L^2(Q))$ and  $g_1,g_2,...,g_m$ is a $\varepsilon c_2$-cover of $(\mathcal{G},L^2(Q))$, then for any $(f,g) \in \mathcal{F} \times \mathcal{G}$, we have:
		$$| f(x)g(x)-f_i(x)g_j(x)| \le |f(x)-f_i(x)|c_2+c_1|g(x)-g_i(x)|.$$
		which implies that $ \{ f_i g_j : 1\le i \le n \text{ and } 1\le j \le m \}$
		is a $2\varepsilon c_1c_2$ -cover of $\mathcal{F} \cdot \mathcal{G}.$
	\end{proof}
		

		The following lemma is just a simplied version result of the theory of VC Hull class of functions (Section 3.6.3 in \cite{Gine2016}).
	
	\begin{lemma}
		\label{construction 2}
		If $f$ is a bounded measurable function on the measurable space $(\R^m,\mathcal{B}(\R^m))$ and $D =[a,b] \subset \mathbb{R}$ is a compact interval, then 
		$$\mathcal{F}:= \{ f +d : d \in D\},$$
		is VC type with respect to a constant envelope.
	\end{lemma}
	\begin{proof}
		Let $N= [ \frac{b-a}{\varepsilon}]$, $f_i = f+i\varepsilon$ for all $i \in [\![ 1,N]\!]$. So, by the definition of $\mathcal{F}$, for all $g \in \mathcal{F}$, there is an $i \in [\![1,N]\!]$ such that $|g(x)-f_i(x)|<\varepsilon$ for all $x \in \R^m$.
		Thus, for all probability measure $Q$ on $\R^m$, we have:
		$\| g- f_i\|_{L^2(Q)} \le \varepsilon, $
		which makes $\mathcal{H}:= \{ f_i  : i \in [\![1,N]\!]\}$ be a $\varepsilon$-cover of $L^2(Q)$. Hence,
		$$N(\varepsilon, \mathcal{F}, L^2(Q)) \le N \stackrel{}{\le} \frac{(b-a)}{\varepsilon}. $$
		So $\mathcal{F}$ is a VC-type class of functions with $A=b-a$, $v=1$,$F=\max(1, \|f\|_{\infty}+|a|,\|f\|_{\infty}+|b|) $.
	\end{proof}

\section{Some estimates using the total variation}

\begin{lemma}	\label{Lemma: an inequality on bounded variation functions}
	If $K: [0, +\infty) \rightarrow \mathbb{R}$ is a bounded variation function with $H(a)$ its total variation on  the interval $[0,a]$, for all $a, b \in [0,\infty]$, with 
	$a\leq b$,
	\begin{equation}
		|K(b)-K(a)| \le H(b) -H(a). \label{Equation: first inequality on K}
	\end{equation}
Besides, if $K$ satisfies Assumption \ref{hyp:K}, then, when $b$ goes to infinity,
\begin{equation}
 K(b)b^{d+3}=o(1)\qquad \text{and } \int_{b}^{\infty} K(a)a^{d+1} \mathrm{d}a=o(1/b). \label{Equation: second inequality on K}
\end{equation}
\end{lemma}
\begin{proof}[Proof of Lemma \ref{Lemma: an inequality on bounded variation functions}]
Inequality \eqref{Equation: first inequality on K} comes directly from the definition of total variation. We note that:
\[b^{d+3}(H(\infty)- H(b)) \le \int_b^{\infty} a^{d+3} \mathrm{d}H(a).\]
Then, by Assumption \ref{hyp:K}, 
\[\lim_{b\rightarrow +\infty} b^{d+3}(H(\infty)- H(b))=0.\]
Then, as:
$$K(b)b^{d+3} \le b^{d+3}(H(\infty)- H(b)),$$
we have proven the first estimation in \eqref{Equation: second inequality on K}.

For the second estimation, we see that:
\begin{align*}
		(d+2)&\int_{b}^{\infty} bK(a)a^{d+1} \mathrm{d}a\le (d+2)\int_{b}^{\infty} b(H(\infty)-H(a))a^{d+1} \mathrm{d}a\\
		 =&-b^{d+3}(H(\infty)-H(b))+b\int_b^{\infty}a^{d+2}\mathrm{d}H(a)
		\\
		 \le& -b^{d+3}(H(\infty)-H(b))+\int_b^{\infty}a^{d+3}\mathrm{d}H(a).
\end{align*}

Therefore, we have the conclusion.
\end{proof}
	
\section{Proof of Lemma \ref{Lemma: Some auxiliary calculations.}}\label{app:proofLemma_aux}

	Thanks to the symmetry of the Euclidean norm $\| \cdot \|_2$, we observe that for any $i,j \in [\![1,d ]\!]$,
	$$\int_{B_{\mathbb{R}^d}(0,c)}  G(\|v\|_2)  v^iv^j \mathrm{d}v= \begin{cases}
		0 &\text{if } i \ne j,\\
		\frac{1}{d}\int_{B_{\mathbb{R}^d}(0,c)}  G(\|v\|_2)  \|v\|_2^2 \mathrm{d}v & \text{if } i = j.
	\end{cases}$$
	Thus, LHS of \eqref{Equation: first equation in the auxiliary lemma} is equal to:
	\begin{align*}
		=&	\left[\frac{1}{d}\int_{B_{\mathbb{R}^d}(0,c)}  G(\|v\|_2)  \|v\|_2^2 \mathrm{d}v \right]  \left[\sum_{i=1}^d \left\langle \nabla_{\mathbb{R}^m}f(x) , \frac{\partial k}{\partial x^i} (0)\right\rangle \left\langle \nabla_{\mathbb{R}^m}h(x) , \frac{\partial k}{\partial x^i} (0)\right\rangle\right]
		\\
		=& \left[\frac{1}{d}\int_{B_{\mathbb{R}^d}(0,c)}  G(\|v\|_2)  \|v\|_2^2 \mathrm{d}v \right]  \left[\sum_{i=1}^d \frac{\partial (f\circ k)}{\partial x^i}(0) \, \frac{\partial (h\circ k)}{\partial x^i}(0)\right]
		\\
		=& \left[ \frac{1}{d}\int_{B_{\mathbb{R}^d}(0,c)}  G(\|v\|_2)  \|v\|_2^2 \mathrm{d}v \right]\big\langle \nabla_{\mathbb{R}^d}( f\circ k )(0), \nabla_{\mathbb{R}^d}( h\circ k )(0)\big\rangle .
	\end{align*}
	Hence, we have  \eqref{Equation: first equation in the auxiliary lemma}.\\
	For \eqref{Equation: second equation in the auxiliary lemma}, for all $i$, thanks again to the symmetry of the Euclidean norm $\| \cdot \|_2$, we have 
	$$\int_{B_{\mathbb{R}^d}(0,c)}  G(\|v\|_2)  v^i \mathrm{d}v=0,$$
	Thus, LHS of  \eqref{Equation: second equation in the auxiliary lemma} is equal to
	\begin{align*}
	 =&\bigg[ \left\langle \nabla_{\mathbb{R}^m} f(x),\frac{1}{2}\sum_{i=1}^d \frac{\partial^2 k}{\partial x^i\partial x^i}(0) \right\rangle 
		+\frac{1}{2}\sum_{i=1}^d f''(x)\left(\frac{\partial k}{\partial x^i}(0),\frac{\partial k}{\partial x^i}(0) \right)\bigg]\times
		\\
		& \qquad \qquad  \qquad\qquad \times\left( \frac{1}{d}\int_{B_{\mathbb{R}^d}(0,c)}  G(\|v\|_2)  \|v\|_2^2 \mathrm{d}v \right).			
	\end{align*}
Besides, since $k(0)=x$,
\begin{align*}
	 &\left\langle \nabla_{\mathbb{R}^m} f(x),\sum_{i=1}^d \frac{\partial^2 k}{\partial x^i\partial x^i}(0) \right\rangle 
	+\sum_{i=1}^d f''(x)\left(\frac{\partial k}{\partial x^i}(0),\frac{\partial k}{\partial x^i}(0) \right)
	\\
	=&  \sum_{i=1}^d \left[ \sum_{j=1}^m \frac{\partial f}{\partial x^j}(x)\frac{\partial^2 k^j}{\partial x^i\partial x^i}(0)+\sum_{j,l=1}^m \frac{\partial^2 f}{ \partial x^j \partial x^l}(x) \frac{ \partial k^j}{\partial x^i}(0)\frac{ \partial k^l}{\partial x^i}(0) \right]
	\\
	=& \sum_{i=1}^d \left[ \sum_{j=1}^m \frac{\partial}{\partial x^i}\left( \frac{\partial f}{\partial x^j} \circ k \times  \frac{\partial k^j}{\partial x^i} \right) \Big|_0 \right]
	\\
	=&  \sum_{i=1}^d \frac{\partial^2 ( f \circ k)}{\partial x^i \partial x^i}(0)=\Delta_{\mathbb{R}^d} ( f \circ k)(0).
\end{align*}
This ends the proof of Lemma \ref{Lemma: Some auxiliary calculations.}.

\end{document}